\documentclass{amsart}
\usepackage[all]{xy}
\usepackage{enumerate}
\usepackage{mathrsfs}
\usepackage{amssymb}
\usepackage{graphicx}

\theoremstyle{plain}
\newtheorem{theorem}{Theorem}[section]
\newtheorem{lemma}[theorem]{Lemma}
\newtheorem{corollary}[theorem]{Corollary}
\newtheorem{proposition}[theorem]{Proposition}
\newtheorem{conjecture}[theorem]{Conjecture}
\theoremstyle{definition}

\theoremstyle{remark}
\newtheorem{remark}[theorem]{Remark}

\newtheorem{problem}[theorem]{Problem}

\newcommand\Aut{\operatorname{Aut}}
\newcommand\Inn{\operatorname{Inn}}
\newcommand\Out{\operatorname{Out}}
\newcommand\Hom{\operatorname{Hom}}

\newcommand\GL{\operatorname{GL}}
\newcommand\IA{\operatorname{IA}}
\newcommand\IO{\operatorname{IO}}

\newcommand\Sym{\operatorname{Sym}}
\newcommand\sym{\operatorname{sym}}
\newcommand\alt{\operatorname{alt}}
\newcommand\Ext{\operatorname{Ext}}

\newcommand\id{\operatorname{id}}

\newcommand\gr{\mathrm{gr}}

\newcommand\sgn{\mathrm{sgn}}
\newcommand\Z{\mathbb{Z}}

\newcommand\Q{\mathbb{Q}}

\newcommand\A{\mathbf{A}}

\newcommand\gpS{\mathfrak{S}}

\newcommand\im{\operatorname{im}}

\newcommand\centre[1]{\begin{array}{c} #1 \end{array}}

\title[Homology of $\IA_n$ with coefficients in Jacobi diagrams]{The first homology of $\IA_n$ with coefficients \\in spaces of Jacobi diagrams}
\author{Mai Katada}

\address{Faculty of Mathematics, Kyushu University, 744, Motooka, Nishi-ku, Fukuoka, Japan 819-0395}
\email{katada@math.kyushu-u.ac.jp}

\keywords{Jacobi diagrams, Automorphism groups of free groups, General linear groups, IA-automorphism groups of free groups}

\subjclass[2020]{20F28, 20J06, 57K16}

\begin{document}

\begin{abstract}
The aim of this paper is to compute the first homology of $\IA_n$ with coefficients in the $\Aut(F_n)$-module $A_2(n)$ of Jacobi diagrams of degree $2$ on $n$-component oriented arcs.
We also prove the self-duality of the $\Aut(F_3)$-submodule $A''_2(3)$ of $A_2(3)$ which is generated by a certain element.
\end{abstract}

\maketitle
\section{Introduction}


The IA-automorphism group $\IA_n$ of the free group $F_n$ of rank $n$ is
defined to be the kernel of the canonical map from the automorphism group $\Aut(F_n)$ of $F_n$ to the general linear group $\GL(n,\Z)$ induced by the abelianization map of $F_n$.
It follows from the short exact sequence of groups
\begin{gather}\label{defofIA}
    1\to \IA_n\to \Aut(F_n)\to \GL(n,\Z)\to 1
\end{gather}
that the rational homology of $\IA_n$ admits a natural $\GL(n,\Z)$-module structure.
By the Hochschild--Serre spectral sequence associated to the short exact sequence \eqref{defofIA}, the rational homology of $\IA_n$ is related to the twisted homology of $\Aut(F_n)$ and $\GL(n,\Z)$.

The (rational) homology of $\IA_n$ has been studied by many authors.
The first homology of $\IA_n$ is completely determined.
There are several papers that have obtained the first homology of $\IA_n$, and the independent results of Cohen--Pakianathan \cite{CP}, Farb \cite{Farb} and Kawazumi \cite{Kawazumi} are thought to be the first. However, before their papers, the first homology of $\IA_n$ was determined by Bachmuth \cite{Bachmuth} as a free abelian group and by Formanek \cite{Formanek} as a $\GL(n,\Z)$-module.
Krstic--McCool \cite{KM} proved that $\IA_3$ is not finitely presentable, and
Bestvina--Bux--Margalit \cite{BBM} proved that $H_2(\IA_3,\Z)$ has infinite rank.
Day--Putman \cite{DP} proved that $H_2(\IA_n,\Z)$ is finitely generated as a $\GL(n,\Z)$-representation.
The Albanese homology of $\IA_n$, which is a quotient $\GL(n,\Z)$-representation of the rational homology of $\IA_n$, is defined as the image of the map induced by the abelianization map of $\IA_n$ on homology.
Pettet \cite{Pettet} determined the second Albanese homology of $\IA_n$.
Satoh \cite{Satoh2021} obtained a subquotient $\GL(3,\Z)$-representation of $H_2(\IA_3,\Q)$ which is not contained in the Albanese homology of $\IA_3$.
We determined the Albanese homology in degree $3$ and studied it in higher degrees \cite{KatadaIA}.
However, even the second rational homology has not been completely determined, and higher degree rational homology of $\IA_n$ is still mysterious.

For any (right) $\Aut(F_n)$-module $M$, the homology of $\IA_n$ with coefficients in $M$ also admits a $\GL(n,\Z)$-module structure.
To the best of our knowledge, however, the homology of $\IA_n$ with any non-trivial coefficients has not been computed.
Any $\Aut(F_n)$-module which does not factorize through a $\GL(n,\Z)$-representation restricts to a non-trivial $\IA_n$-module.
The homology of $\IA_n$ with coefficients in such a module is related to the homology of $\Aut(F_n)$ with coefficients in it by the Hochschild--Serre spectral sequence.
As for the cohomology of $\Aut(F_n)$ with coefficients in such a module, the only study we know is by Satoh \cite{SatohFricke}.
He proved the non-triviality of the first cohomology of $\Aut(F_n)$ with coefficients in the right $\Aut(F_n)$-module $\Hom_{\Q}( J/J^2, J^2/J^3)$, where $J$ is the ideal of the ring of \emph{Fricke characters} that is generated by $\operatorname{tr} x-2$ for $x\in F_n$.

Among such coefficients, the easiest one to study may be what we call polynomial $\Aut(F_n)$-modules, which are obtained from algebraic $\GL(n,\Z)$-representations by extensions.
In this paper, we will study the homology of $\IA_n$ with coefficients in
the polynomial $\Aut(F_n)$-module $A_2(n)$
of \emph{Jacobi diagrams} of degree $2$ on $n$-component oriented arcs.
The right $\Aut(F_n)$-module structure on the spaces $A_d(n)$ of Jacobi diagrams of degree $d$ was introduced and studied in \cite{Katada1, Katada2}. (See Section \ref{subsecJacobidiagrams} for details.)

In \cite{Katada1}, we obtained a direct decomposition of $\Aut(F_n)$-modules
$$A_2(n)=A'_2(n)\oplus A''_2(n),$$ 
where $A'_2(n)$ is irreducible and $A''_2(n)$ is indecomposable.
The indecomposable submodule $A''_2(n)$ uniquely admits the following composition series
\begin{gather*}
 A''_2(n)\supset A_{2,1}(n)\supset A_{2,2}(n)\supset 0,
\end{gather*}
where $A_{2,k}(n)$ is the subspace of $A''_2(n)$ consisting of Jacobi diagrams with at least $k$ trivalent vertices for $k=1,2$.
Therefore, we have
\begin{gather*}
    H_1(\IA_n,A_2(n))\cong H_1(\IA_n,A'_2(n))\oplus H_1(\IA_n,A''_2(n)),
\end{gather*}
and we can easily compute the $\GL(n,\Z)$-module structure of $H_1(\IA_n,A'_2(n))$ since $\IA_n$ acts trivially on $A'_2(n)$.
Hence, we have only to study the $\GL(n,\Z)$-module structure of $H_1(\IA_n,A''_2(n))$.
The following is one of our main theorems.

\begin{theorem}[Theorems \ref{thmH1A2''} and \ref{thmH1A2''3}]
 We have
\begin{gather*}
  \begin{split}
        H_1(\IA_n, A''_2(n))\cong 
        \begin{cases}
        V_{1^2,2^21}\oplus V_{1^2,32}\oplus V_{1,21^2}\oplus V_{1,2^2}^{\oplus 2} \oplus V_{1,31}\oplus V_{0,21}^{\oplus 2} & n\ge 5\\
        V_{1^2,32}\oplus V_{1,21^2}\oplus V_{1,2^2}^{\oplus 2} \oplus V_{1,31}\oplus V_{0,21}^{\oplus 2} & n= 4\\
        V_{1,2^2}\oplus V_{1,31}\oplus V_{0,21}^{\oplus 2} & n=3,
        \end{cases}
  \end{split} 
\end{gather*}
where $V_{\lambda^+,\lambda^-}$ denotes the irreducible algebraic $\GL(n,\Z)$-representation corresponding to a pair $(\lambda^+,\lambda^-)$ of partitions.
\end{theorem}

We also compute the first homology of $\IA_n$ with coefficients in the other non-trivial $\IA_n$-modules $A''_2(n)/A_{2,2}(n)$ and $A_{2,1}(n)$.

A topological meaning of the homology of $\IA_n$ with coefficients in the spaces of Jacobi diagrams is as follows.
The degree completion of the spaces of Jacobi diagrams $A_d(n)$ appears as the target space of the \emph{Kontsevich integral} for $n$-component \emph{bottom tangles}.
We have an action of the \emph{handlebody group} $\mathcal{H}_n$
on the space $B(n)$ of $n$-component bottom tangles \cite{HM}.
The action of $\Aut(F_n)$ on $A_d(n)$ can be interpreted as a restriction of this action of $\mathcal{H}_n$ on $B(n)$ via the surjective group homomorphism from $\mathcal{H}_n$ to $\Aut(F_n)$ that is induced by the action of $\mathcal{H}_n$ on the fundamental group of the handlebody \cite{Katada1}.
Therefore, the homology of $\Aut(F_n)$ with coefficients in $A_d(n)$ can be regarded as an approximation of the homology of $\mathcal{H}_n$ with coefficients in $B(n)$.
Moreover, the homology of $\IA_n$ with coefficients in $A_d(n)$ can be regarded as an approximation of the homology of the kernel of the composition map
$\mathcal{H}_n\twoheadrightarrow \Aut(F_n)\twoheadrightarrow \GL(n,\Z)$
with coefficients in $B(n)$.

The second aim of this paper is to prove that the $\Aut(F_3)$-module $A''_2(3)$ is self-dual and that $A''_2(3)/A_{2,2}(3)$ is isomorphic to $A_{2,1}(3)^*$, which we conjectured in \cite[Remark 7.13]{Katada1}.

\begin{theorem}[Theorem \ref{selfduality} and Corollary \ref{dualA21andA''2/A22}]\label{introthmdual}
    We have an isomorphism of $\Aut(F_3)$-modules
    \begin{gather*}
        \Phi: A''_2(3)\to A''_2(3)^*,
    \end{gather*}
    which induces an isomorphism of $\Aut(F_3)$-modules
    \begin{gather*}
        \overline{\Phi}: A''_2(3)/A_{2,2}(3)\to A_{2,1}(3)^*.
    \end{gather*}
\end{theorem}

By using the Hochschild--Pirashvili homology (or higher Hochschild homology) on the wedges of circles and the coalgebra of dual numbers, Turchin--Willwacher \cite{TW} constructed an $\Out(F_n)$-module $U^{I}_3$, which does not factorize through a $\GL(n,\Z)$-representation.
It follows from the work of Kielak \cite{Kielak} that for $n=3$, their module $U^{I}_3$ gives the first example of such $\Out(F_3)$-modules of the smallest dimension $7$.
In \cite{Katada1}, we proved that the $\Aut(F_n)$-action on $A_d(n)$ induces an action of the outer automorphism group $\Out(F_n)$ of $F_n$.
The dimension of $A_{2,1}(3)$ and $A''_2(3)/A_{2,2}(3)$ are both $7$, and $A_{2,1}(3)$ is isomorphic to $U^{I}_3$ as $\Out(F_3)$-modules.
Theorem \ref{introthmdual} implies that $A''_2(3)/A_{2,2}(3)$ is isomorphic to the dual of $U^{I}_3$ as $\Out(F_3)$-modules.

\subsection{Outline}
 The rest of the paper is organized as follows.
 In Section \ref{secPreliminaries}, we will recall algebraic $\GL(n,\Q)$-representations and the $\Aut(F_n)$-module structure of the spaces of Jacobi diagrams.
 In Section \ref{secPreliminariesofhomology}, we will give irreducible decomposition of the first homology of $\IA_n$ with coefficients in trivial $\IA_n$-modules.
 In Section \ref{secH1A2''/A22}, we will compute the first homology of $\IA_n$ with coefficients in $A''_2(n)/A_{2,2}(n)$.
 In Sections \ref{secH1A21} and \ref{secH1A''2}, we will compute the first homology of $\IA_n$ for $n\ge 4$ with coefficients in $A_{2,1}(n)$ and $A''_2(n)$, respectively.
 In Section \ref{secselfdual}, we will prove that $A''_2(3)$ is a self-dual $\Aut(F_3)$-module.
 In Section \ref{secH1A''23}, we will compute the first homology of $\IA_3$ with coefficients in $A''_2(3)$.
 In Section \ref{secH2}, we will consider the second and higher degree homology of $\IA_n$ with coefficients in $A_2(n)$.
 In Section \ref{secH1IO}, we will study the homology of a subgroup $\IO_n$ of $\Out(F_n)$, which is analogous to $\IA_n$, with coefficients in $A_2(n)$.
 In Section \ref{secperspectives}, we will consider the cases of the spaces of Jacobi diagrams of degree not smaller than $3$. 

\subsection*{Acknowledgements}
The author would like to thank Kazuo Habiro for valuable advice. 
She also thanks Takao Satoh for conversation about topics around homology of $\IA_n$ and $\Aut(F_n)$.
This work was supported by JSPS KAKENHI Grant Number JP22KJ1864.

\section{Preliminaries}\label{secPreliminaries}

Let $n\ge 1$.
Let $F_n=\langle x_1,\dots,x_n\rangle$ be the free group of rank $n$.
Let $H=H(n)=H_1(F_n,\Q)$ and fix a basis $\{e_i\}_{i=1}^{n}$ for $H$.
We take the dual basis $\{e_i^*\}_{i=1}^{n}$ for $H^*=H(n)^*\cong H^1(F_n,\Q)$. 

Let $\Aut(F_n)$ denote the automorphism group of $F_n$, and $\GL(n,\Z)$ the general linear group over $\Z$.
The abelianization of $F_n$ induces a group homomorphism from $\Aut(F_n)$ to $\GL(n,\Z)$, whose kernel is called the IA-automorphism group of $F_n$ and denoted by $\IA_n$.

Throughout this paper, we will consider vector spaces over $\Q$, and $G$-modules mean $\Q[G]$-modules for $G=\IA_n, \Aut(F_n)$ and $\GL(n,\Z)$.

In this section, we will recall some facts from representation theory of the general linear group over $\Q$ and the $\Aut(F_n)$-module structure of the spaces of Jacobi diagrams.

\subsection{Algebraic $\GL(n,\Q)$-representations}
Here we will recall some notions from representation theory. See Fulton--Harris \cite{FH} and Koike \cite{Koike} for details.

A finite dimensional $\GL(n,\Q)$-representation $V$ is \emph{algebraic} if after choosing a basis for $V$, the $(\dim V)^2$ coordinate functions of the action $\GL(n,\Q)\to \GL(V)$ are rational functions of the $n^2$ variables.

It is well known that algebraic $\GL(n,\Q)$-representations are completely reducible, and that irreducible algebraic $\GL(n,\Q)$-representations are classified by \emph{bipartitions}, that is, pairs of partitions.
Here, a \emph{partition} $\lambda=(\lambda_1,\dots,\lambda_n)$ with at most $n$ parts is a sequence of non-negative integers such that $\lambda_1\ge \lambda_2\ge \dots\ge \lambda_n$.

The irreducible algebraic $\GL(n,\Q)$-representation $V_{\underline\lambda}=V_{\underline{\lambda}}(n)$
corresponding to a bipartition $\underline\lambda=(\lambda^+,\lambda^-)$ can be constructed as follows.

Let $l(\lambda)=\max(\{0\}\cup \{i\mid \lambda_i> 0\})$ denote the \emph{length} of $\lambda$ and $|\lambda|=\sum_{i=1}^{l(\lambda)}\lambda_i$ the \emph{size} of $\lambda$.
A \emph{Young diagram} of $\lambda$ consists of $\lambda_i$ boxes in the $i$-th row so that the rows of boxes are left-aligned.
A \emph{tableau} on a Young diagram is a bijection from the set of integers $\{1,\dots, |\lambda|\}$ to the set of the boxes.
The \emph{canonical tableau} is the tableau whose numbering starts from the first row from left to right and followed by the second row from left to right and so on.
A tableau is called \emph{standard} if the numbering is strictly increasing in each row and each column.
The Young symmetrizer $c_{\lambda}\in \Q[\gpS_{|\lambda|}]$ is defined by
\begin{gather*}
    c_{\lambda}=\sum_{\substack{\sigma\in R\\\tau\in C}} \sgn(\tau)\tau\sigma,
\end{gather*}
where $R\subset \gpS_{|\lambda|}$ (resp. $C\subset \gpS_{|\lambda|}$) is the subgroup consisting of elements of $\gpS_{|\lambda|}$ preserving rows (resp. columns) of the canonical tableau on the Young diagram corresponding to $\lambda$.
Let $S^{\lambda}=\Q[\gpS_{|\lambda|}] c_{\lambda}$ denote the \emph{Specht module} corresponding to $\lambda$.

Consider $H$ as the standard representation of $\GL(n,\Q)$ and $H^*$ as the dual representation.
Let $p$ and $q$ be positive integers.
For $k\in \{1,\dots, p\},l\in \{1,\dots,q\}$, the \emph{contraction map} 
\begin{gather*}
    c_{k,l}: H^{\otimes p}\otimes (H^*)^{\otimes q} \to H^{\otimes p-1}\otimes (H^*)^{\otimes q-1}
\end{gather*}
is defined by 
\begin{gather*}
    c_{k,l}(v_1\otimes \cdots\otimes v_p\otimes f_1\otimes \cdots\otimes f_q)=\langle v_k,f_l\rangle v_1\otimes \cdots \hat{v}_k \cdots\otimes v_p\otimes f_1\otimes\cdots\hat{f}_l \cdots\otimes f_q,
\end{gather*}
where $\langle -,- \rangle: H\otimes H^*\to \Q$ denotes the dual pairing, and where $\hat{v}_k$ and $\hat{f}_l$ denote the omission of $v_k$ and $f_l$.
The \emph{traceless part} $T^{p,q}\subset H^{\otimes p}\otimes (H^*)^{\otimes q}$ is defined by
\begin{gather*}
    T^{p,q}=\bigcap_{\substack{k\in \{1,\dots, p\}\\ l\in \{1,\dots, q\}}} \ker c_{k,l}.
\end{gather*}
Define
$$V_{\underline\lambda}=T^{|\lambda^+|,|\lambda^-|}\otimes_{\Q[\gpS_{|\lambda^+|}\times \gpS_{|\lambda^-|}]} S^{\lambda^+}\otimes S^{\lambda^-}.$$
If $n\ge l(\lambda^+)+l(\lambda^-)$, then $V_{\underline\lambda}$ is irreducible and otherwise we have $V_{\underline\lambda}=0$. 
Note that we have $V_{1,0}=H$ and $V_{0,1}=H^*$.

For two bipartitions $\underline{\lambda}, \underline{\mu}$ such that $n\ge l(\lambda^+)+l(\lambda^-)+l(\mu^+)+l(\mu^-)$, we have the following formula of the irreducible decomposition of the tensor product of $V_{\underline{\lambda}}$ and $V_{\underline{\mu}}$
\begin{gather}\label{Koike}
V_{\underline{\lambda}}\otimes V_{\underline{\mu}}\cong \bigoplus_{\underline{\nu}} V_{\underline{\nu}}^{\oplus N_{\underline{\lambda} \underline{\mu}}^{\underline{\nu}}},\quad
N_{\underline{\lambda} \underline{\mu}}^{\underline{\nu}}=
\sum_{\alpha\beta\theta\delta}
(\sum_{\kappa}N_{\kappa\alpha}^{\lambda^+}N_{\kappa\beta}^{\mu^-})
(\sum_{\epsilon}N_{\epsilon\theta}^{\lambda^-}N_{\epsilon\delta}^{\mu^+})N_{\alpha\delta}^{\nu^+}N_{\beta\theta}^{\nu^-},
\end{gather}
where $N_{\lambda\mu}^{\nu}$ is the Littlewood--Richardson coefficient (see \cite{Koike} for details).

An \emph{algebraic $\GL(n,\Z)$-representation} is the restriction of an algebraic $\GL(n,\Q)$-representation to $\GL(n,\Z)$.

Note that any subrepresentations, quotient representations, direct sums and tensor products of algebraic $\GL(n,\Z)$-representations are algebraic.
Moreover, algebraic $\GL(n,\Z)$-representations are closed under extensions.

\begin{lemma}\label{splitalgebraic}
Let $n\ge 3$.
Suppose that we have a short exact sequence of $\GL(n,\Z)$-representations
    \begin{gather*}
        0\to M'\to M\to M''\to 0,
    \end{gather*}
    where $M'$ and $M''$ are algebraic $\GL(n,\Z)$-representations.
    Then the above short exact sequence splits and we have $M\cong M'\oplus M''$, which implies that $M$ is also an algebraic $\GL(n,\Z)$-representation.
\end{lemma}

\begin{proof}
    We have $\Ext^1_{\GL(n,\Z)}(M'',M')\cong H^1(\GL(n,\Z),(M'')^*\otimes M')$.
    Since $M'$ and $M''$ are algebraic, 
    we have $H^1(\GL(n,\Z),(M'')^*\otimes M')=0$ for $n\ge 3$ by the work of Borel \cite{Borel1, Borel2}, Li-Sun \cite{Li-Sun} (cf. Theorem \ref{thmBorel}).
\end{proof}

\begin{remark}
    By Lindell \cite{LindellK}, we have a stronger statement than Lemma \ref{splitalgebraic}, that is, extensions of any finite dimensional $\GL(n,\Z)$-representations split.
\end{remark}

\subsection{The $\Aut(F_n)$-module $A_d(n)$ of Jacobi diagrams}\label{subsecJacobidiagrams}
Here we recall the vector spaces of Jacobi diagrams and an action of $\Aut(F_n)$ on the spaces of Jacobi diagrams, which was studied in \cite{Katada1, Katada2}.

A \emph{Jacobi diagram} $D$ on $n$-component oriented arcs $X_n$ is a uni-trivalent graph $D$ such that each trivalent vertex is equipped with a cyclic order, each connected component has at least one univalent vertex and univalent vertices are embedded into the interior of $X_n$.
Two Jacobi diagrams are regarded as the same if there is a homeomorphism between them whose restriction to the arc components is isotopic to the identity map.
The degree of a Jacobi diagram is defined as the half of the number of its vertices.

Let $A_d(n)$ denote the vector space spanned by Jacobi diagrams of degree $d$ on $X_n$ modulo the STU relations:
\begin{gather*}
\scalebox{0.7}{$\centre{
\begingroup%
  \makeatletter%
  \providecommand\color[2][]{%
    \errmessage{(Inkscape) Color is used for the text in Inkscape, but the package 'color.sty' is not loaded}%
    \renewcommand\color[2][]{}%
  }%
  \providecommand\transparent[1]{%
    \errmessage{(Inkscape) Transparency is used (non-zero) for the text in Inkscape, but the package 'transparent.sty' is not loaded}%
    \renewcommand\transparent[1]{}%
  }%
  \providecommand\rotatebox[2]{#2}%
  \newcommand*\fsize{\dimexpr\f@size pt\relax}%
  \newcommand*\lineheight[1]{\fontsize{\fsize}{#1\fsize}\selectfont}%
  \ifx\svgwidth\undefined%
    \setlength{\unitlength}{64.04390103bp}%
    \ifx\svgscale\undefined%
      \relax%
    \else%
      \setlength{\unitlength}{\unitlength * \real{\svgscale}}%
    \fi%
  \else%
    \setlength{\unitlength}{\svgwidth}%
  \fi%
  \global\let\svgwidth\undefined%
  \global\let\svgscale\undefined%
  \makeatother%
  \begin{picture}(1,0.67292281)%
    \lineheight{1}%
    \setlength\tabcolsep{0pt}%
    \put(0,0){\includegraphics[width=\unitlength,page=1]{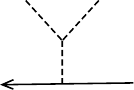}}%
  \end{picture}%
\endgroup%
}$}=\scalebox{0.7}{$\centre{
\begingroup%
  \makeatletter%
  \providecommand\color[2][]{%
    \errmessage{(Inkscape) Color is used for the text in Inkscape, but the package 'color.sty' is not loaded}%
    \renewcommand\color[2][]{}%
  }%
  \providecommand\transparent[1]{%
    \errmessage{(Inkscape) Transparency is used (non-zero) for the text in Inkscape, but the package 'transparent.sty' is not loaded}%
    \renewcommand\transparent[1]{}%
  }%
  \providecommand\rotatebox[2]{#2}%
  \newcommand*\fsize{\dimexpr\f@size pt\relax}%
  \newcommand*\lineheight[1]{\fontsize{\fsize}{#1\fsize}\selectfont}%
  \ifx\svgwidth\undefined%
    \setlength{\unitlength}{64.04525894bp}%
    \ifx\svgscale\undefined%
      \relax%
    \else%
      \setlength{\unitlength}{\unitlength * \real{\svgscale}}%
    \fi%
  \else%
    \setlength{\unitlength}{\svgwidth}%
  \fi%
  \global\let\svgwidth\undefined%
  \global\let\svgscale\undefined%
  \makeatother%
  \begin{picture}(1,0.67044537)%
    \lineheight{1}%
    \setlength\tabcolsep{0pt}%
    \put(0,0){\includegraphics[width=\unitlength,page=1]{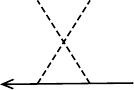}}%
  \end{picture}%
\endgroup%
}$}-\scalebox{0.7}{$\centre{
\begingroup%
  \makeatletter%
  \providecommand\color[2][]{%
    \errmessage{(Inkscape) Color is used for the text in Inkscape, but the package 'color.sty' is not loaded}%
    \renewcommand\color[2][]{}%
  }%
  \providecommand\transparent[1]{%
    \errmessage{(Inkscape) Transparency is used (non-zero) for the text in Inkscape, but the package 'transparent.sty' is not loaded}%
    \renewcommand\transparent[1]{}%
  }%
  \providecommand\rotatebox[2]{#2}%
  \newcommand*\fsize{\dimexpr\f@size pt\relax}%
  \newcommand*\lineheight[1]{\fontsize{\fsize}{#1\fsize}\selectfont}%
  \ifx\svgwidth\undefined%
    \setlength{\unitlength}{64.04503507bp}%
    \ifx\svgscale\undefined%
      \relax%
    \else%
      \setlength{\unitlength}{\unitlength * \real{\svgscale}}%
    \fi%
  \else%
    \setlength{\unitlength}{\svgwidth}%
  \fi%
  \global\let\svgwidth\undefined%
  \global\let\svgscale\undefined%
  \makeatother%
  \begin{picture}(1,0.65917634)%
    \lineheight{1}%
    \setlength\tabcolsep{0pt}%
    \put(0,0){\includegraphics[width=\unitlength,page=1]{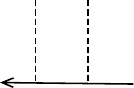}}%
  \end{picture}%
\endgroup%
}$}.
\end{gather*}

Then the vector space $A_d(n)$ admits the following right $\Aut(F_n)$-module structure.
An automorphism $f: F_n\to F_n$ can be identified with an $n$-component oriented arcs mapped into a handlebody of genus $n$ in such a way that the endpoints of the arcs are uniformly distributed on the bottom line and the $i$-th component goes from the $2i$-th point to the $(2i-1)$-st point, where we count the endpoints from left to right.
We will write the oriented arcs in the handlebody corresponding to the automorphism $f$ as $Z(f)$.
For example, the automorphism
\begin{gather}\label{exampleU12}
    f : F_3\to F_3, \quad x_1\mapsto x_1x_2, \; x_2\mapsto x_2,\; x_3\mapsto x_3
\end{gather}
is identified with 
\begin{gather*}
    Z(f)=\centre{
\begingroup%
  \makeatletter%
  \providecommand\color[2][]{%
    \errmessage{(Inkscape) Color is used for the text in Inkscape, but the package 'color.sty' is not loaded}%
    \renewcommand\color[2][]{}%
  }%
  \providecommand\transparent[1]{%
    \errmessage{(Inkscape) Transparency is used (non-zero) for the text in Inkscape, but the package 'transparent.sty' is not loaded}%
    \renewcommand\transparent[1]{}%
  }%
  \providecommand\rotatebox[2]{#2}%
  \newcommand*\fsize{\dimexpr\f@size pt\relax}%
  \newcommand*\lineheight[1]{\fontsize{\fsize}{#1\fsize}\selectfont}%
  \ifx\svgwidth\undefined%
    \setlength{\unitlength}{75.751171bp}%
    \ifx\svgscale\undefined%
      \relax%
    \else%
      \setlength{\unitlength}{\unitlength * \real{\svgscale}}%
    \fi%
  \else%
    \setlength{\unitlength}{\svgwidth}%
  \fi%
  \global\let\svgwidth\undefined%
  \global\let\svgscale\undefined%
  \makeatother%
  \begin{picture}(1,0.7094533)%
    \lineheight{1}%
    \setlength\tabcolsep{0pt}%
    \put(0,0){\includegraphics[width=\unitlength,page=1]{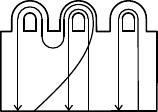}}%
  \end{picture}%
\endgroup%
}.
\end{gather*}
The action of $f:F_n\to F_n$ on a Jacobi diagram $u\in A_d(n)$ is defined to be the composition of $u$ and the corresponding diagram $Z(f)$ in the category $\A$ of Jacobi diagrams in handlebodies, which was introduced in \cite{HM}.
The composition $Z(f)\circ u$ is obtained by stacking a suitable cabling of $u$ on the top of the handlebody of $Z(f)$. (See \cite[Section 2.2]{Katada1} for details.)
For example, the action of the automorphism $f$ defined in \eqref{exampleU12} on the Jacobi diagram 
$u=\scalebox{0.7}{$\centre{
\begingroup%
  \makeatletter%
  \providecommand\color[2][]{%
    \errmessage{(Inkscape) Color is used for the text in Inkscape, but the package 'color.sty' is not loaded}%
    \renewcommand\color[2][]{}%
  }%
  \providecommand\transparent[1]{%
    \errmessage{(Inkscape) Transparency is used (non-zero) for the text in Inkscape, but the package 'transparent.sty' is not loaded}%
    \renewcommand\transparent[1]{}%
  }%
  \providecommand\rotatebox[2]{#2}%
  \newcommand*\fsize{\dimexpr\f@size pt\relax}%
  \newcommand*\lineheight[1]{\fontsize{\fsize}{#1\fsize}\selectfont}%
  \ifx\svgwidth\undefined%
    \setlength{\unitlength}{80.22484197bp}%
    \ifx\svgscale\undefined%
      \relax%
    \else%
      \setlength{\unitlength}{\unitlength * \real{\svgscale}}%
    \fi%
  \else%
    \setlength{\unitlength}{\svgwidth}%
  \fi%
  \global\let\svgwidth\undefined%
  \global\let\svgscale\undefined%
  \makeatother%
  \begin{picture}(1,0.38851671)%
    \lineheight{1}%
    \setlength\tabcolsep{0pt}%
    \put(0,0){\includegraphics[width=\unitlength,page=1]{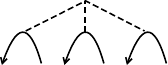}}%
  \end{picture}%
\endgroup%
}$}$
is given by
\begin{gather*}
    u\cdot f=\scalebox{1.0}{$\centre{
\begingroup%
  \makeatletter%
  \providecommand\color[2][]{%
    \errmessage{(Inkscape) Color is used for the text in Inkscape, but the package 'color.sty' is not loaded}%
    \renewcommand\color[2][]{}%
  }%
  \providecommand\transparent[1]{%
    \errmessage{(Inkscape) Transparency is used (non-zero) for the text in Inkscape, but the package 'transparent.sty' is not loaded}%
    \renewcommand\transparent[1]{}%
  }%
  \providecommand\rotatebox[2]{#2}%
  \newcommand*\fsize{\dimexpr\f@size pt\relax}%
  \newcommand*\lineheight[1]{\fontsize{\fsize}{#1\fsize}\selectfont}%
  \ifx\svgwidth\undefined%
    \setlength{\unitlength}{73.7176313bp}%
    \ifx\svgscale\undefined%
      \relax%
    \else%
      \setlength{\unitlength}{\unitlength * \real{\svgscale}}%
    \fi%
  \else%
    \setlength{\unitlength}{\svgwidth}%
  \fi%
  \global\let\svgwidth\undefined%
  \global\let\svgscale\undefined%
  \makeatother%
  \begin{picture}(1,0.42281189)%
    \lineheight{1}%
    \setlength\tabcolsep{0pt}%
    \put(0,0){\includegraphics[width=\unitlength,page=1]{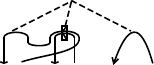}}%
  \end{picture}%
\endgroup%
}$}=\scalebox{0.7}{$\centre{}$}-\frac{1}{2}\scalebox{0.7}{$\centre{
\begingroup%
  \makeatletter%
  \providecommand\color[2][]{%
    \errmessage{(Inkscape) Color is used for the text in Inkscape, but the package 'color.sty' is not loaded}%
    \renewcommand\color[2][]{}%
  }%
  \providecommand\transparent[1]{%
    \errmessage{(Inkscape) Transparency is used (non-zero) for the text in Inkscape, but the package 'transparent.sty' is not loaded}%
    \renewcommand\transparent[1]{}%
  }%
  \providecommand\rotatebox[2]{#2}%
  \newcommand*\fsize{\dimexpr\f@size pt\relax}%
  \newcommand*\lineheight[1]{\fontsize{\fsize}{#1\fsize}\selectfont}%
  \ifx\svgwidth\undefined%
    \setlength{\unitlength}{80.22484197bp}%
    \ifx\svgscale\undefined%
      \relax%
    \else%
      \setlength{\unitlength}{\unitlength * \real{\svgscale}}%
    \fi%
  \else%
    \setlength{\unitlength}{\svgwidth}%
  \fi%
  \global\let\svgwidth\undefined%
  \global\let\svgscale\undefined%
  \makeatother%
  \begin{picture}(1,0.46256357)%
    \lineheight{1}%
    \setlength\tabcolsep{0pt}%
    \put(0,0){\includegraphics[width=\unitlength,page=1]{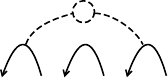}}%
  \end{picture}%
\endgroup%
}$}.
\end{gather*}

The space $A_d(n)$ of Jacobi diagrams admits a descending filtration of finite length
\begin{gather*}
    A_d(n)=A_{d,0}(n)\supset A_{d,1}(n)\supset A_{d,2}(n)\supset \dots \supset A_{d,2d-2}(n)\supset A_{d,2d-1}(n)=0,
\end{gather*}
where $A_{d,k}(n)$ is spanned by Jacobi diagrams of degree $d$ with at least $k$ trivalent vertices.
The $\Aut(F_n)$-action preserves this filtration $A_{d,*}(n)$. 

We will use the following results about the $\Aut(F_n)$-action on $A_d(n)$ (see \cite[Theorem 5.1]{Katada1} and \cite[Lemma 5.4]{Katada1}).

\begin{lemma}[\cite{Katada1}]\label{Innerautomorphism}
The inner automorphism group $\Inn(F_n)$ of $F_n$ acts trivially on $A_d(n)$ for $d,n\ge 0$.
\end{lemma}

\begin{lemma}[\cite{Katada1}]\label{bracketmap}
The $\Aut(F_n)$-action induces the bracket map 
$$[,]: A_{d,k}(n)\times \IA_n\to A_{d,k+1}(n)$$
defined by $[u,f]=u\cdot f-u$ for $u\in A_{d,k}(n)$ and $f\in \IA_n$.    
\end{lemma}

\subsection{The $\Aut(F_n)$-module $A_2(n)$}

Here we briefly recall the $\Aut(F_n)$-module structure of $A_2(n)$. 
See \cite{Katada1} for details.
Let $A'_2(n)$ and $A''_2(n)$ be the $\Aut(F_n)$-submodules of $A_2(n)$ generated by the elements $P'$ and $P''$, respectively, described below:
\begin{gather*}
       P'=  \scalebox{1.0}{$\centre{
\begingroup%
  \makeatletter%
  \providecommand\color[2][]{%
    \errmessage{(Inkscape) Color is used for the text in Inkscape, but the package 'color.sty' is not loaded}%
    \renewcommand\color[2][]{}%
  }%
  \providecommand\transparent[1]{%
    \errmessage{(Inkscape) Transparency is used (non-zero) for the text in Inkscape, but the package 'transparent.sty' is not loaded}%
    \renewcommand\transparent[1]{}%
  }%
  \providecommand\rotatebox[2]{#2}%
  \newcommand*\fsize{\dimexpr\f@size pt\relax}%
  \newcommand*\lineheight[1]{\fontsize{\fsize}{#1\fsize}\selectfont}%
  \ifx\svgwidth\undefined%
    \setlength{\unitlength}{38.25117926bp}%
    \ifx\svgscale\undefined%
      \relax%
    \else%
      \setlength{\unitlength}{\unitlength * \real{\svgscale}}%
    \fi%
  \else%
    \setlength{\unitlength}{\svgwidth}%
  \fi%
  \global\let\svgwidth\undefined%
  \global\let\svgscale\undefined%
  \makeatother%
  \begin{picture}(1,0.81065931)%
    \lineheight{1}%
    \setlength\tabcolsep{0pt}%
    \put(0,0){\includegraphics[width=\unitlength,page=1]{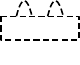}}%
    \put(0.18962287,0.4102936){\makebox(0,0)[lt]{\lineheight{1.45000005}\smash{\begin{tabular}[t]{l}$\sym_{4}$\end{tabular}}}}%
    \put(0,0){\includegraphics[width=\unitlength,page=2]{P20.pdf}}%
  \end{picture}%
\endgroup%
}$},\quad
       P''= \scalebox{1.0}{$\centre{
\begingroup%
  \makeatletter%
  \providecommand\color[2][]{%
    \errmessage{(Inkscape) Color is used for the text in Inkscape, but the package 'color.sty' is not loaded}%
    \renewcommand\color[2][]{}%
  }%
  \providecommand\transparent[1]{%
    \errmessage{(Inkscape) Transparency is used (non-zero) for the text in Inkscape, but the package 'transparent.sty' is not loaded}%
    \renewcommand\transparent[1]{}%
  }%
  \providecommand\rotatebox[2]{#2}%
  \newcommand*\fsize{\dimexpr\f@size pt\relax}%
  \newcommand*\lineheight[1]{\fontsize{\fsize}{#1\fsize}\selectfont}%
  \ifx\svgwidth\undefined%
    \setlength{\unitlength}{38.23996689bp}%
    \ifx\svgscale\undefined%
      \relax%
    \else%
      \setlength{\unitlength}{\unitlength * \real{\svgscale}}%
    \fi%
  \else%
    \setlength{\unitlength}{\svgwidth}%
  \fi%
  \global\let\svgwidth\undefined%
  \global\let\svgscale\undefined%
  \makeatother%
  \begin{picture}(1,0.81249545)%
    \lineheight{1}%
    \setlength\tabcolsep{0pt}%
    \put(0,0){\includegraphics[width=\unitlength,page=1]{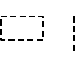}}%
    \put(0.05597441,0.37984023){\makebox(0,0)[lt]{\lineheight{1.45000005}\smash{\begin{tabular}[t]{l}$\alt_2$\end{tabular}}}}%
    \put(0,0){\includegraphics[width=\unitlength,page=2]{Q20.pdf}}%
  \end{picture}%
\endgroup%
}$}\in A_2(4),
\end{gather*}
where 
$$\sym_4=\sum_{\sigma\in \gpS_4}\sigma,\quad \alt_2=\sum_{\tau\in \gpS_2}\sgn(\tau)\tau.$$

\begin{theorem}[\cite{Katada1}]
For $n\ge 3$, we have
\begin{gather*}
    A_2(n)=A'_2(n)\oplus A''_2(n),
\end{gather*}
where $A'_2(n)\cong V_{0,4}$ is an irreducible $\Aut(F_n)$-module, 
and where $A''_2(n)$ is an indecomposable $\Aut(F_n)$-module with the unique composition series
\begin{gather*}
    A''_2(n)\supsetneq A_{2,1}(n)\supsetneq A_{2,2}(n)\supsetneq 0,
\end{gather*}
whose composition factors are
\begin{gather*}
        A''_2(n)/A_{2,1}(n)\cong V_{0,2^2},\quad
        A_{2,1}(n)/A_{2,2}(n)\cong V_{0,1^3},\quad
        A_{2,2}(n)\cong V_{0,2}.
\end{gather*}
\end{theorem}

\section{Preliminaries to compute $H_1(\IA_n,A_2(n))$}\label{secPreliminariesofhomology}

Let $M$ be a right $\Aut(F_n)$-module.
Then the homology $H_*(\IA_n,M)$ of $\IA_n$ with coefficients in $M$ has a $\GL(n,\Z)$-module structure.
(See \cite{Brown} for fundamental facts on homology of groups.)
If the $\IA_n$-action on $M$ is trivial, then we have the following isomorphism of $\GL(n,\Z)$-representations
\begin{gather*}
    H_*(\IA_n,M)\cong M\otimes H_*(\IA_n,\Q).
\end{gather*}

\subsection{The first rational homology of $\IA_n$}

By Formanek \cite{Formanek}, we have the following isomorphism of $\GL(n,\Z)$-representations
\begin{gather*}
    H_1(\IA_n,\Z)\cong \Hom(H_{\Z},{\bigwedge}^2 H_{\Z})\cong H^*_{\Z}\otimes{\bigwedge}^2 H_{\Z}, \quad H_{\Z}=H_1(F_n,\Z).
\end{gather*}
By tensoring with $\Q$, we obtain the $\GL(n,\Z)$-module structure of the first rational homology of $\IA_n$
\begin{gather*}
    H_1(\IA_n,\Q)\cong \Hom(H,{\bigwedge}^2 H)\cong H^*\otimes{\bigwedge}^2 H.
\end{gather*}

\subsection{Irreducible decompositions of the first homology of $\IA_n$ with coefficients in trivial $\IA_n$-modules}\label{irdecompH1}

Note that $\IA_n$ acts trivially on the spaces $A'_2(n)$, $A''_2(n)/A_{2,1}(n)$, $A_{2,1}(n)/A_{2,2}(n)$ and $A_{2,2}(n)$.
Therefore, by using the irreducible decomposition formula \eqref{Koike}, we can easily compute decompositions of the first homology of $\IA_n$ with coefficients in them into irreducible algebraic $\GL(n,\Z)$-representations as follows:
\begin{gather}\label{decompH1A2'}
\begin{gathered}
       H_1(\IA_n,A'_2(n))\cong A'_2(n)\otimes H_1(\IA_n,\Q)
       \cong V_{0,4}\otimes (H^*\otimes {\bigwedge}^2 H)\\
       \cong 
       \begin{cases}
        V_{1^2,41}\oplus V_{1^2,5}\oplus V_{1,4}^{\oplus 2}\oplus V_{1,31}\oplus V_{0,3} & n\ge 4\\
        V_{1^2,5}\oplus V_{1,4}^{\oplus 2}\oplus V_{1,31}\oplus V_{0,3} & n=3,
       \end{cases}
\end{gathered}
\end{gather}
\begin{gather}\label{decompH1A''2/A21}
\begin{gathered}
    H_1(\IA_n, A''_2(n)/A_{2,1}(n))
    \cong A''_2(n)/A_{2,1}(n)\otimes H_1(\IA_n,\Q)
    \cong V_{0,2^2}\otimes(H^*\otimes {\bigwedge}^2 H)\\
    \quad\cong 
    \begin{cases}
        V_{1^2,2^21}\oplus V_{1^2,32}\oplus V_{1,21^2}\oplus V_{1,2^2}^{\oplus 2} \oplus V_{1,31}\oplus V_{0,1^3}\oplus V_{0,21}^{\oplus 2} & n\ge 5\\
        V_{1^2,32}\oplus V_{1,21^2}\oplus V_{1,2^2}^{\oplus 2} \oplus V_{1,31}\oplus V_{0,1^3}\oplus V_{0,21}^{\oplus 2} & n= 4\\
        V_{1,2^2} \oplus V_{1,31}\oplus V_{0,1^3}\oplus V_{0,21}^{\oplus 2} & n=3,
    \end{cases}
\end{gathered}
\end{gather}
\begin{gather}\label{decompH1A21/A22}
\begin{gathered}
      H_1(\IA_n, A_{2,1}(n)/A_{2,2}(n))
      \cong A_{2,1}(n)/A_{2,2}(n)\otimes H_1(\IA_n, \Q)
      \cong V_{0,1^3}\otimes (H^*\otimes {\bigwedge}^2 H)\\
      \quad\cong 
      \begin{cases}
          V_{1^2,1^4}\oplus V_{1^2,21^2}\oplus V_{1,1^3}^{\oplus 2}\oplus V_{1,21}\oplus V_{0,1^2}^{\oplus 2}\oplus V_{0,2} & n\ge 6\\
          V_{1^2,21^2}\oplus V_{1,1^3}^{\oplus 2}\oplus V_{1,21}\oplus V_{0,1^2}^{\oplus 2}\oplus V_{0,2} & n=5\\
          V_{1,1^3}\oplus V_{1,21}\oplus V_{0,1^2}^{\oplus 2}\oplus V_{0,2} & n=4\\
          V_{0,1^2}\oplus V_{0,2} & n=3,
      \end{cases}    
\end{gathered}
\end{gather}
\begin{gather}\label{decompH1A22}
 \begin{gathered}
        H_1(\IA_n,A_{2,2}(n))\cong A_{2,2}(n) \otimes H_1(\IA_n,\Q)
        \cong V_{0,2}\otimes (H^*\otimes {\bigwedge}^2 H)\\
        \cong 
        \begin{cases}
            V_{1^2,21}\oplus V_{1^2,3}\oplus V_{1,1^2}\oplus V_{1,2}^{\oplus 2}\oplus V_{0,1} & n\ge 4\\
            V_{1^2,3}\oplus V_{1,1^2}\oplus V_{1,2}^{\oplus 2}\oplus V_{0,1} & n=3.
        \end{cases}     
 \end{gathered}
\end{gather}

\subsection{Contraction maps}\label{subseccontraction}

Here we will define several \emph{contraction maps}, which detect irreducible components of $H_1(\IA_n,A_{2,1}(n)/A_{2,2}(n))$ and $H_1(\IA_n,A_{2,2}(n))$.

Define $\GL(n,\Z)$-homomorphisms
\begin{gather*}
\iota:{\bigwedge}^3 H^*\otimes H^*\otimes {\bigwedge}^2 H\hookrightarrow (H^*)^{\otimes 4}\otimes H^{\otimes 2},\quad 
\iota':{\Sym}^2 H^*\otimes H^*\otimes {\bigwedge}^2 H\hookrightarrow (H^*)^{\otimes 3}\otimes H^{\otimes 2}
\end{gather*}
by
\begin{gather*}
\begin{split}
      \iota((a_1^*\wedge a_2^*\wedge a_3^*)\otimes a^*\otimes (b_1\wedge b_2))&=\iota^{(1^3)}(a_1^*\wedge a_2^*\wedge a_3^*)\otimes a^*\otimes \iota^{(1^2)}(b_1\wedge b_2),\\
      \iota'((a_1^*\cdot a_2^*)\otimes a^*\otimes (b_1\wedge b_2))&=\iota^{(2)}(a_1^*\cdot a_2^*)\otimes a^*\otimes \iota^{(1^2)}(b_1\wedge b_2),
\end{split}
\end{gather*}
for $a_1,a_2,a_3,a,b_1,b_2\in H$, where $\iota^{(1^2)}$, $\iota^{(2)}$ $\iota^{(1^3)}$ are the canonical injective $\GL(n,\Z)$-homomorphisms defined by
\begin{gather*}
    \begin{split}
    &\iota^{(1^2)}: {\bigwedge}^2 H\hookrightarrow H^{\otimes 2},\quad
    b_1\wedge b_2\mapsto \frac{1}{2}(b_1\otimes b_2-b_2\otimes b_1),\\
    &\iota^{(2)}: {\Sym}^2 H^*\hookrightarrow (H^*)^{\otimes 2},\quad
    a_1^*\cdot a_2^*\mapsto \frac{1}{2}(a_1^*\otimes a_2^*+a_2^*\otimes a_1^*),\\
    &\iota^{(1^3)}: {\bigwedge}^3 H^*\hookrightarrow (H^*)^{\otimes 3},\quad
    a_1^*\wedge a_2^* \wedge a_3^*\mapsto 
    \frac{1}{6}\sum_{\sigma\in \gpS_3}\sgn(\sigma)a_{\sigma(1)}^*\otimes a_{\sigma(2)}^*\otimes a_{\sigma(3)}^*
    \end{split}
\end{gather*}
for $a_1,a_2,a_3, b_1,b_2\in H$.

In order to detect irreducible components of $H_1(\IA_n,A_{2,1}(n)/A_{2,2}(n))$, we will use the following $7$ contraction maps $c_{1^2,1^4}, c_{1^2,21^2}, c_{1,1^3}, c_{1,1^3}', c_{1,21}, c_{0,1^2}, c_{0,1^2}'$ defined by
\begin{gather*}
    \begin{split}
        c_{1^2,1^4}&: (H^*)^{\otimes 4}\otimes H^{\otimes 2}\to {\bigwedge}^4 H^*\otimes{\bigwedge}^2 H,\\
        &a_1^*\otimes a_2^*\otimes a_3^*\otimes a_4^*\otimes b_1\otimes b_2\mapsto (a_1^*\wedge a_2^*\wedge a_3^*\wedge a_4^*)\otimes (b_1\wedge b_2),\\
        c_{1^2,21^2}&: (H^*)^{\otimes 4}\otimes H^{\otimes 2}\to {\bigwedge}^3 H^*\otimes H^* \otimes{\bigwedge}^2 H,\\
        &a_1^*\otimes a_2^*\otimes a_3^*\otimes a_4^*\otimes b_1\otimes b_2\mapsto (a_1^*\wedge a_2^*\wedge a_3^*)\otimes a_4^*\otimes (b_1\wedge b_2),\\
         c_{1,1^3}&: (H^*)^{\otimes 4}\otimes H^{\otimes 2}\to {\bigwedge}^3 H^*\otimes H,\\
        &a_1^*\otimes a_2^*\otimes a_3^*\otimes a_4^*\otimes b_1\otimes b_2\mapsto a_4^*(b_2)(a_1^*\wedge a_2^*\wedge a_3^*)\otimes b_1,\\
        c_{1,1^3}'&: (H^*)^{\otimes 4}\otimes H^{\otimes 2}\to {\bigwedge}^3 H^*\otimes H,\\
        &a_1^*\otimes a_2^*\otimes a_3^*\otimes a_4^*\otimes b_1\otimes b_2\mapsto a_1^*(b_2)(a_2^*\wedge a_3^*\wedge a_4^*)\otimes b_1,\\
        c_{1,21}&: (H^*)^{\otimes 4}\otimes H^{\otimes 2}\to {\bigwedge}^2 H^*\otimes H^*\otimes H,\\
        &a_1^*\otimes a_2^*\otimes a_3^*\otimes a_4^*\otimes b_1\otimes b_2\mapsto a_1^*(b_2)(a_2^*\wedge a_3^*)\otimes a_4^*\otimes b_1,\\
        c_{0,1^2}&: (H^*)^{\otimes 4}\otimes H^{\otimes 2}\to {\bigwedge}^2 H^*,\\
        &a_1^*\otimes a_2^*\otimes a_3^*\otimes a_4^*\otimes b_1\otimes b_2\mapsto a_4^*(b_2)a_1^*(b_1)(a_2^*\wedge a_3^*),\\
        c_{0,1^2}'&: (H^*)^{\otimes 4}\otimes H^{\otimes 2}\to {\bigwedge}^2 H^*,\\
        &a_1^*\otimes a_2^*\otimes a_3^*\otimes a_4^*\otimes b_1\otimes b_2\mapsto a_1^*(b_1)a_2^*(b_2)(a_3^*\wedge a_4^*)
    \end{split}
\end{gather*}
for $a_1,a_2,a_3,a_4,b_1,b_2\in H$.

In a similar way, to detect irreducible components of $H_1(\IA_n,A_{2,2}(n))$, we will use the following $6$ contraction maps $c_{1^2,21}, c_{1^2,3}, c_{1,1^2}, c_{1,2}, c'_{1,2}, c_{0,1}$ defined by
\begin{gather*}
    \begin{split}
        c_{1^2,21}&: (H^*)^{\otimes 3}\otimes H^{\otimes 2}\to H^*\otimes {\bigwedge}^2 H^*\otimes{\bigwedge}^2 H,\\
        &a_1^*\otimes a_2^*\otimes a_3^*\otimes b_1\otimes b_2\mapsto a_1^*\otimes (a_2^*\wedge a_3^*)\otimes (b_1\wedge b_2),\\
        c_{1^2,3}&: (H^*)^{\otimes 3}\otimes H^{\otimes 2}\to  (H^*)^{\otimes 3}\otimes{\bigwedge}^2 H,\\
        &a_1^*\otimes a_2^*\otimes a_3^*\otimes b_1\otimes b_2\mapsto a_1^*\otimes a_2^*\otimes a_3^*\otimes (b_1\wedge b_2),\\
        c_{1,1^2}&: (H^*)^{\otimes 3}\otimes H^{\otimes 2}\to {\bigwedge}^2 H^*\otimes H,\\
        &a_1^*\otimes a_2^*\otimes a_3^*\otimes b_1\otimes b_2\mapsto a_2^*(b_1) (a_1^*\wedge a_3^*)\otimes b_2,\\
        c_{1,2}&: (H^*)^{\otimes 3}\otimes H^{\otimes 2}\to (H^*)^{\otimes 2}\otimes H,\\
        &a_1^*\otimes a_2^*\otimes a_3^*\otimes b_1\otimes b_2\mapsto a_2^*(b_1) (a_1^*\otimes a_3^*)\otimes b_2,\\
        c'_{1,2}&: (H^*)^{\otimes 3}\otimes H^{\otimes 2}\to (H^*)^{\otimes 2}\otimes H,\\
        &a_1^*\otimes a_2^*\otimes a_3^*\otimes b_1\otimes b_2\mapsto a_3^*(b_2) (a_1^*\otimes a_2^*)\otimes b_1,\\
        c_{0,1}&: (H^*)^{\otimes 3}\otimes H^{\otimes 2}\to H^*,\\
        &a_1^*\otimes a_2^*\otimes a_3^*\otimes b_1\otimes b_2\mapsto a_2^*(b_1)a_3^*(b_2) a_1^*,\\
    \end{split}
\end{gather*}
for $a_1,a_2,a_3,b_1,b_2\in H$.

\subsection{Abelian $2$-cycles of $\IA_n$}\label{subsecAbeliancycle}

Here, we recall the definition of abelian cycles of groups, and two types of abelian cycles of $\IA_n$ that appeared in \cite[Section 3]{KatadaIA}.

For $i\ge 1$, let $\{\phi_1,\dots, \phi_i\}$ be an $i$-tuple of mutually commuting elements of $\IA_n$.
Then we have a group homomorphism $\phi: \Z^i\to \IA_n$ sending the element $z_j$ of a $\Z$-basis for $\Z^i$ to $\phi_j$ for $j\in\{1,\dots,i\}$.
The group homomorphism $\phi$ induces a group homomorphism $\phi_*: H_i(\Z^i,\Q)\to H_i(\IA_n,\Q)$.
The \emph{abelian cycle} corresponding to the $i$-tuple $\{\phi_1,\dots,\phi_i\}$ is defined to be the image of the fundamental class in $H_i(\Z^i,\Q)$ under $\phi_*$.

Magnus \cite{Magnus} discovered the following finite set of generators of $\IA_n$:
\begin{gather*}
    \{g_{a,b}\mid 1\le a,b\le n, a\neq b\}\cup\{f_{a,b,c}\mid 1\le a,b,c\le n, a<b, a\neq c\neq b\},
\end{gather*}
where $g_{a,b}$ and $f_{a,b,c}$ are defined by
\begin{alignat*}{2}
    g_{a,b}(x_b)&= x_a x_b x_a^{-1},&\quad g_{a,b}(x_d)&= x_d\quad (d\neq b),\\
    f_{a,b,c}(x_c)&=x_c[x_a,x_b],& \quad f_{a,b,c}(x_d)&=x_d\quad (d\neq c).
\end{alignat*}

If $n\ge 3$, then we have a pair of mutually commuting elements $g_{2,1}$ and $g_{3,1}g_{3,2}$.
Let $\alpha_{0,2}$ denote the corresponding abelian cycle, that is, $$\alpha_{0,2}=[g_{2,1}\otimes g_{3,1}g_{3,2}-g_{3,1}g_{3,2}\otimes g_{2,1}]\in H_2(\IA_n,\Q).$$

If $n\ge 4$, then we have a pair of mutually commuting elements $g_{2,1}$ and $g_{4,3}$. 
Let $\alpha_{0,1^2}$ denote the corresponding abelian cycle, that is, $$\alpha_{0,1^2}=[g_{2,1}\otimes g_{4,3}-g_{4,3}\otimes g_{2,1}]\in H_2(\IA_n,\Q).$$

\subsection{Notations}

In the following sections, we will use the following matrices in order to study  $\GL(n,\Z)$-module structures.

Let $\id\in \GL(n,\Z)$ denote the identity matrix.
For distinct elements $k,l\in \{1,\dots, n\}$, let $E_{k,l}\in \GL(n,\Z)$ denote the matrix that maps $e_l$ to $e_l+e_k$ and fixes $e_a$ for $a\neq l$.
Note that $E_{k,l}$ acts on $H^*$ by sending $e_k^*$ to $e_k^*-e_l^*$ and fixing $e_a^*$ for $a\neq k$.
Let $P_{k,l}$ denote the permutation matrix that exchanges $e_k$ and $e_l$, and fixes $e_a$ for $a\neq k,l$.

\section{The first homology of $\IA_n$ with coefficients in $A''_2(n)/A_{2,2}(n)$}\label{secH1A2''/A22}

In this section, we will study the homology $H_1(\IA_n,A''_2(n)/A_{2,2}(n))$.
We will determine the $\GL(n,\Z)$-module structure by using the long exact sequences associated to short exact sequences of $\Aut(F_n)$-modules.

\subsection{Two long exact sequences of homology}

The short exact sequence
\begin{gather*}
    0\to A_{2,2}(n)\to A''_2(n)\to A''_2(n)/A_{2,2}(n)\to 0
\end{gather*}
induces the long exact sequence
\begin{gather}\label{H1A2''}
\begin{split}
    \dots \to H_1(\IA_n, A''_2(n))&\to 
    H_1(\IA_n, A''_2(n)/A_{2,2}(n))
    \to
    H_0(\IA_n, A_{2,2}(n))\\
    &\to  
    H_0(\IA_n, A''_2(n))\to 
    H_0(\IA_n, A''_2(n)/A_{2,2}(n))\to 0.
\end{split}
\end{gather}
Since we have
\begin{gather*}
\begin{split}
    H_0(\IA_n, A_{2,2}(n))&=(A_{2,2}(n))_{\IA_n}\cong A_{2,2}(n)\cong V_{0,2},\\
    H_0(\IA_n, A''_2(n))&=(A''_2(n))_{\IA_n}\cong A''_2(n)/A_{2,1}(n),\\
    H_0(\IA_n, A''_2(n)/A_{2,2}(n))&=(A''_2(n)/A_{2,2}(n))_{\IA_n}\cong A''_2(n)/A_{2,1}(n),
\end{split}
\end{gather*}
the long exact sequence \eqref{H1A2''} yields the exact sequence
\begin{gather}\label{H1A2''new}
    \dots \to H_1(\IA_n, A''_2(n))\to H_1(\IA_n, A''_2(n)/A_{2,2}(n))
    \to V_{0,2}\to 0.
\end{gather}

On the other hand, the short exact sequence
\begin{gather*}
    0\to A_{2,1}(n)/A_{2,2}(n)\to A''_2(n)/A_{2,2}(n)\to A''_2(n)/A_{2,1}(n)\to 0
\end{gather*}
induces the long exact sequence
\begin{gather}\label{H1A2''/A22}
\begin{split}
    &\quad\quad\quad\quad\quad\quad\quad\cdots\quad\quad\quad\quad\:\:\:\to 
    H_2(\IA_n, A''_2(n)/A_{2,2}(n))\to 
    H_2(\IA_n, A''_2(n)/A_{2,1}(n))\\
    &\xrightarrow{\partial_2} 
    H_1(\IA_n, A_{2,1}(n)/A_{2,2}(n))\to  
    H_1(\IA_n, A''_2(n)/A_{2,2}(n))\to 
    H_1(\IA_n, A''_2(n)/A_{2,1}(n))\\
    &\xrightarrow{\partial_1}
    H_0(\IA_n, A_{2,1}(n)/A_{2,2}(n))\to  
    H_0(\IA_n, A''_2(n)/A_{2,2}(n))\to 
    H_0(\IA_n, A''_2(n)/A_{2,1}(n))\to 0.
\end{split}
\end{gather}
Here, we have
\begin{gather*}
    H_0(\IA_n, A_{2,1}(n)/A_{2,2}(n))\cong A_{2,1}(n)/A_{2,2}(n)\cong V_{0,1^3},\\
    H_0(\IA_n, A''_2(n)/A_{2,2}(n))\cong
    H_0(\IA_n, A''_2(n)/A_{2,1}(n))\cong A''_2(n)/A_{2,1}(n).
\end{gather*}
Therefore, we have 
\begin{gather}\label{imagepartial1}
    \im(\partial_1)\cong V_{0,1^3}.
\end{gather}

\subsection{The image of the boundary map $\partial_2$}

Here, we will detect a $\GL(n,\Z)$-subrepresentation of $H_1(\IA_n,A_{2,1}(n)/A_{2,2}(n))$ which is contained in the image of the boundary map $\partial_2$.

\begin{proposition}\label{imageboundary2}
    We have 
    \begin{gather*}
        \im (\partial_2) \supset
        \begin{cases}
          V_{1^2,1^4}\oplus V_{1^2,21^2}\oplus V_{1,1^3}^{\oplus 2}\oplus V_{1,21}\oplus V_{0,1^2}^{\oplus 2} & n\ge 6\\
          V_{1^2,21^2}\oplus V_{1,1^3}^{\oplus 2}\oplus V_{1,21}\oplus V_{0,1^2}^{\oplus 2} & n=5\\
          V_{1,1^3}\oplus V_{1,21}\oplus V_{0,1^2}^{\oplus 2} & n=4\\
          V_{0,1^2} & n=3.
        \end{cases}
    \end{gather*}
\end{proposition}

\begin{proof}
 For $n\ge 3$, set
 \begin{gather*}
 v:=\scalebox{0.5}{$\centre{
\begingroup%
  \makeatletter%
  \providecommand\color[2][]{%
    \errmessage{(Inkscape) Color is used for the text in Inkscape, but the package 'color.sty' is not loaded}%
    \renewcommand\color[2][]{}%
  }%
  \providecommand\transparent[1]{%
    \errmessage{(Inkscape) Transparency is used (non-zero) for the text in Inkscape, but the package 'transparent.sty' is not loaded}%
    \renewcommand\transparent[1]{}%
  }%
  \providecommand\rotatebox[2]{#2}%
  \newcommand*\fsize{\dimexpr\f@size pt\relax}%
  \newcommand*\lineheight[1]{\fontsize{\fsize}{#1\fsize}\selectfont}%
  \ifx\svgwidth\undefined%
    \setlength{\unitlength}{158.97484197bp}%
    \ifx\svgscale\undefined%
      \relax%
    \else%
      \setlength{\unitlength}{\unitlength * \real{\svgscale}}%
    \fi%
  \else%
    \setlength{\unitlength}{\svgwidth}%
  \fi%
  \global\let\svgwidth\undefined%
  \global\let\svgscale\undefined%
  \makeatother%
  \begin{picture}(1,0.19050601)%
    \lineheight{1}%
    \setlength\tabcolsep{0pt}%
    \put(0,0){\includegraphics[width=\unitlength,page=1]{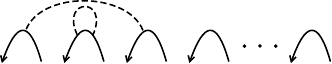}}%
  \end{picture}%
\endgroup%
}$}-\;\frac{1}{2}\scalebox{0.5}{$\centre{
\begingroup%
  \makeatletter%
  \providecommand\color[2][]{%
    \errmessage{(Inkscape) Color is used for the text in Inkscape, but the package 'color.sty' is not loaded}%
    \renewcommand\color[2][]{}%
  }%
  \providecommand\transparent[1]{%
    \errmessage{(Inkscape) Transparency is used (non-zero) for the text in Inkscape, but the package 'transparent.sty' is not loaded}%
    \renewcommand\transparent[1]{}%
  }%
  \providecommand\rotatebox[2]{#2}%
  \newcommand*\fsize{\dimexpr\f@size pt\relax}%
  \newcommand*\lineheight[1]{\fontsize{\fsize}{#1\fsize}\selectfont}%
  \ifx\svgwidth\undefined%
    \setlength{\unitlength}{158.97484198bp}%
    \ifx\svgscale\undefined%
      \relax%
    \else%
      \setlength{\unitlength}{\unitlength * \real{\svgscale}}%
    \fi%
  \else%
    \setlength{\unitlength}{\svgwidth}%
  \fi%
  \global\let\svgwidth\undefined%
  \global\let\svgscale\undefined%
  \makeatother%
  \begin{picture}(1,0.14609354)%
    \lineheight{1}%
    \setlength\tabcolsep{0pt}%
    \put(0,0){\includegraphics[width=\unitlength,page=1]{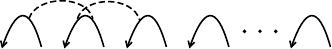}}%
  \end{picture}%
\endgroup%
}$}-\;\frac{1}{2}\scalebox{0.5}{$\centre{
\begingroup%
  \makeatletter%
  \providecommand\color[2][]{%
    \errmessage{(Inkscape) Color is used for the text in Inkscape, but the package 'color.sty' is not loaded}%
    \renewcommand\color[2][]{}%
  }%
  \providecommand\transparent[1]{%
    \errmessage{(Inkscape) Transparency is used (non-zero) for the text in Inkscape, but the package 'transparent.sty' is not loaded}%
    \renewcommand\transparent[1]{}%
  }%
  \providecommand\rotatebox[2]{#2}%
  \newcommand*\fsize{\dimexpr\f@size pt\relax}%
  \newcommand*\lineheight[1]{\fontsize{\fsize}{#1\fsize}\selectfont}%
  \ifx\svgwidth\undefined%
    \setlength{\unitlength}{158.97484198bp}%
    \ifx\svgscale\undefined%
      \relax%
    \else%
      \setlength{\unitlength}{\unitlength * \real{\svgscale}}%
    \fi%
  \else%
    \setlength{\unitlength}{\svgwidth}%
  \fi%
  \global\let\svgwidth\undefined%
  \global\let\svgscale\undefined%
  \makeatother%
  \begin{picture}(1,0.14375739)%
    \lineheight{1}%
    \setlength\tabcolsep{0pt}%
    \put(0,0){\includegraphics[width=\unitlength,page=1]{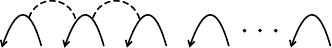}}%
  \end{picture}%
\endgroup%
}$}\in A_2(n).
 \end{gather*}
 Then we have
 \begin{gather*}
     \begin{split}
         \partial_2([v]_{A_{2,1}(n)}\otimes \alpha_{0,2})
         &= [[v, g_{2,1}]]_{A_{2,2}(n)}\otimes [g_{3,1}g_{3,2}]_{[\IA_n,\IA_n]}\\
         &=3(e_1^*\wedge e_2^*\wedge e_3^*)\otimes(e_1^*\otimes (e_3\wedge e_1)+e_2^*\otimes (e_3\wedge e_2))\\
         &\in H_1(\IA_n,A_{2,1}(n)/A_{2,2}(n)),
     \end{split}
 \end{gather*}
 where $[,]$ is the bracket map appeared in Lemma \ref{bracketmap}.
 If $n\ge 3$, then $\im (\partial_2)$ contains $V_{0,1^2}$ since we have
 \begin{gather*}
     c_{0,1^2}\;\iota(\partial_2([v]_{A_{2,1}(n)}\otimes \alpha_{0,2}))= e_1^*\wedge e_2^*.
 \end{gather*}
 If $n\ge 4$, then $\im (\partial_2)$ contains $V_{1,1^3}\oplus V_{1,21}$ since we have
\begin{gather*}
 \begin{split}
     (E_{n,3}-\id)c_{1,1^3}\;\iota(\partial_2([v]_{A_{2,1}(n)}\otimes \alpha_{0,2}))
     &=3(e_1^*\wedge e_2^*\wedge e_3^*)\otimes e_n,\\
     (E_{n,1}-\id) c_{1,21}\;\iota(\partial_2([v]_{A_{2,1}(n)}\otimes \alpha_{0,2}))
     &=-\frac{1}{2}(e_1^*\wedge e_2^*)\otimes e_1^*\otimes e_n.\\
 \end{split}
 \end{gather*}
 If $n\ge 5$, then $\im (\partial_2)$ contains $V_{1^2,21^2}$ since we have
 \begin{gather*}
 \begin{split}
     (E_{n-1,3}-\id)(E_{n,1}-\id)c_{1^2,21^2} \;\iota(\partial_2([v]_{A_{2,1}(n)}\otimes \alpha_{0,2}))
     &=3(e_1^*\wedge e_2^*\wedge e_3^*)\otimes e_1^*\otimes (e_{n-1}\wedge e_n).
 \end{split}
 \end{gather*}

 In order to detect $V_{1^2,1^4}$ for $n\ge 6$ and the multiplicities of $V_{1,1^3}$ for $n\ge 5$ and of $V_{0,1^2}$ for $n\ge 4$, we set
 \begin{gather*}
 v':=\scalebox{0.5}{$\centre{
\begingroup%
  \makeatletter%
  \providecommand\color[2][]{%
    \errmessage{(Inkscape) Color is used for the text in Inkscape, but the package 'color.sty' is not loaded}%
    \renewcommand\color[2][]{}%
  }%
  \providecommand\transparent[1]{%
    \errmessage{(Inkscape) Transparency is used (non-zero) for the text in Inkscape, but the package 'transparent.sty' is not loaded}%
    \renewcommand\transparent[1]{}%
  }%
  \providecommand\rotatebox[2]{#2}%
  \newcommand*\fsize{\dimexpr\f@size pt\relax}%
  \newcommand*\lineheight[1]{\fontsize{\fsize}{#1\fsize}\selectfont}%
  \ifx\svgwidth\undefined%
    \setlength{\unitlength}{188.97484239bp}%
    \ifx\svgscale\undefined%
      \relax%
    \else%
      \setlength{\unitlength}{\unitlength * \real{\svgscale}}%
    \fi%
  \else%
    \setlength{\unitlength}{\svgwidth}%
  \fi%
  \global\let\svgwidth\undefined%
  \global\let\svgscale\undefined%
  \makeatother%
  \begin{picture}(1,0.15831268)%
    \lineheight{1}%
    \setlength\tabcolsep{0pt}%
    \put(0,0){\includegraphics[width=\unitlength,page=1]{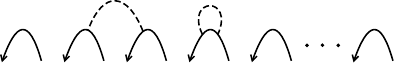}}%
  \end{picture}%
\endgroup%
}$}-\;\frac{1}{2}\scalebox{0.5}{$\centre{
\begingroup%
  \makeatletter%
  \providecommand\color[2][]{%
    \errmessage{(Inkscape) Color is used for the text in Inkscape, but the package 'color.sty' is not loaded}%
    \renewcommand\color[2][]{}%
  }%
  \providecommand\transparent[1]{%
    \errmessage{(Inkscape) Transparency is used (non-zero) for the text in Inkscape, but the package 'transparent.sty' is not loaded}%
    \renewcommand\transparent[1]{}%
  }%
  \providecommand\rotatebox[2]{#2}%
  \newcommand*\fsize{\dimexpr\f@size pt\relax}%
  \newcommand*\lineheight[1]{\fontsize{\fsize}{#1\fsize}\selectfont}%
  \ifx\svgwidth\undefined%
    \setlength{\unitlength}{188.97484154bp}%
    \ifx\svgscale\undefined%
      \relax%
    \else%
      \setlength{\unitlength}{\unitlength * \real{\svgscale}}%
    \fi%
  \else%
    \setlength{\unitlength}{\svgwidth}%
  \fi%
  \global\let\svgwidth\undefined%
  \global\let\svgscale\undefined%
  \makeatother%
  \begin{picture}(1,0.14279739)%
    \lineheight{1}%
    \setlength\tabcolsep{0pt}%
    \put(0,0){\includegraphics[width=\unitlength,page=1]{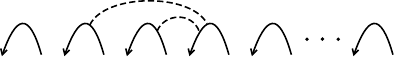}}%
  \end{picture}%
\endgroup%
}$}-\;\frac{1}{2}\scalebox{0.5}{$\centre{
\begingroup%
  \makeatletter%
  \providecommand\color[2][]{%
    \errmessage{(Inkscape) Color is used for the text in Inkscape, but the package 'color.sty' is not loaded}%
    \renewcommand\color[2][]{}%
  }%
  \providecommand\transparent[1]{%
    \errmessage{(Inkscape) Transparency is used (non-zero) for the text in Inkscape, but the package 'transparent.sty' is not loaded}%
    \renewcommand\transparent[1]{}%
  }%
  \providecommand\rotatebox[2]{#2}%
  \newcommand*\fsize{\dimexpr\f@size pt\relax}%
  \newcommand*\lineheight[1]{\fontsize{\fsize}{#1\fsize}\selectfont}%
  \ifx\svgwidth\undefined%
    \setlength{\unitlength}{188.97484239bp}%
    \ifx\svgscale\undefined%
      \relax%
    \else%
      \setlength{\unitlength}{\unitlength * \real{\svgscale}}%
    \fi%
  \else%
    \setlength{\unitlength}{\svgwidth}%
  \fi%
  \global\let\svgwidth\undefined%
  \global\let\svgscale\undefined%
  \makeatother%
  \begin{picture}(1,0.13339946)%
    \lineheight{1}%
    \setlength\tabcolsep{0pt}%
    \put(0,0){\includegraphics[width=\unitlength,page=1]{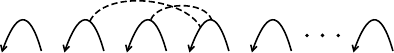}}%
  \end{picture}%
\endgroup%
}$}\in A_2(n)
 \end{gather*}
 for $n\ge 4$.
 Then we have
 \begin{gather*}
     \begin{split}
         \partial_2([v']_{A_{2,1}(n)}\otimes \alpha_{0,1^2})
         &=-[[v', g_{4,3}]]_{A_{2,2}(n)}\otimes [g_{2,1}]_{[\IA_n,\IA_n]}\\
         &=-3 (e_2^*\wedge e_3^*\wedge e_4^*)\otimes e_1^*\otimes (e_2\wedge e_1)\\
         &\in H_1(\IA_n,A_{2,1}(n)/A_{2,2}(n)).
     \end{split}
 \end{gather*}
 If $n\ge 6$, then $\im (\partial_2)$ contains $V_{1^2,1^4}$ since we have
  \begin{gather*}
     (E_{n-1,2}-\id)(E_{n,1}-\id)c_{1^2,1^4} \;\iota(\partial_2([v']_{A_{2,1}(n)}\otimes \alpha_{0,1^2}))
     =3(e_1^*\wedge e_2^*\wedge e_3^*\wedge e_4^*)\otimes (e_{n-1}\wedge e_n).
 \end{gather*}
 If $n\ge 4$, then $\im (\partial_2)$ contains $V_{0,1^2}^{\oplus 2}$ since we have
\begin{alignat*}{2}
     c_{0,1^2}\;\iota(\partial_2([v]_{A_{2,1}(n)}\otimes \alpha_{0,2}))&= e_1^*\wedge e_2^*,&\quad
     c_{0,1^2}'\;\iota(\partial_2([v]_{A_{2,1}(n)}\otimes \alpha_{0,2}))&=- e_1^*\wedge e_2^*,\\
     c_{0,1^2}\;\iota(\partial_2([v']_{A_{2,1}(n)}\otimes \alpha_{0,1^2}))&=-\frac{1}{2}e_3^*\wedge e_4^*,&\quad
     c_{0,1^2}'\;\iota(\partial_2([v']_{A_{2,1}(n)}\otimes \alpha_{0,1^2}))&=0.
\end{alignat*}
 If $n\ge 5$, then $\im (\partial_2)$ contains $V_{1,1^3}^{\oplus 2}$ since we have
 \begin{gather*}
    \begin{split}
     (E_{n,3}-\id)c_{1,1^3}\;\iota(\partial_2([v]_{A_{2,1}(n)}\otimes \alpha_{0,2}))
     &=3(e_1^*\wedge e_2^*\wedge e_3^*)\otimes e_n,\\
     (E_{n,3}-\id)c_{1,1^3}'\;\iota(\partial_2([v]_{A_{2,1}(n)}\otimes \alpha_{0,2}))
     &=(e_1^*\wedge e_2^*\wedge e_3^*)\otimes e_n,\\
     P_{1,4}(E_{n,2}-\id)c_{1,1^3}\;\iota(\partial_2([v']_{A_{2,1}(n)}\otimes \alpha_{0,1^2}))
     &=-\frac{3}{2}(e_1^*\wedge e_2^*\wedge e_3^*)\otimes e_n,\\
     P_{1,4}(E_{n,2}-\id)c'_{1,1^3}\;\iota(\partial_2([v']_{A_{2,1}(n)}\otimes \alpha_{0,1^2}))
     &=0.
    \end{split}
 \end{gather*}
 This completes the proof.
\end{proof}

\subsection{Determination of $H_1(\IA_n,A''_2(n)/A_{2,2}(n))$}

By combining the above results, we will determine the $\GL(n,\Z)$-module structure of the homology of $\IA_n$ with coefficients in $A''_2(n)/A_{2,2}(n)$.

\begin{theorem}\label{thmH1A''2/A22}
We have
    \begin{gather*}
    \begin{split}
       &H_1(\IA_n,A''_2(n)/A_{2,2}(n))\\
       &\cong
        \begin{cases}
        V_{1^2,2^21}\oplus V_{1^2,32}\oplus V_{1,21^2}\oplus V_{1,2^2}^{\oplus 2} \oplus V_{1,31}
        \oplus V_{0,21}^{\oplus 2}\oplus V_{0,2} & n\ge 5\\
        V_{1^2,32}\oplus V_{1,21^2}\oplus V_{1,2^2}^{\oplus 2} \oplus V_{1,31}\oplus V_{0,21}^{\oplus 2}\oplus V_{0,2} & n= 4\\
        V_{1,2^2} \oplus V_{1,31}\oplus V_{0,21}^{\oplus 2}\oplus V_{0,2} & n=3.
        \end{cases}
    \end{split}
    \end{gather*}
\end{theorem}

\newcommand\Cok{\operatorname{Cok}}
\begin{proof}
    By the long exact sequence \eqref{H1A2''/A22}, we have
    \begin{gather}\label{H1A''2/A22}
        0\to \Cok(\partial_2)\to H_1(\IA_n,A''_2(n)/A_{2,2}(n))\to \ker(\partial_1)\to 0.
    \end{gather}
    Since $\Cok(\partial_2)$ and $\ker(\partial_1)$ are algebraic $\GL(n,\Z)$-representations, by Lemma \ref{splitalgebraic}, the exact sequence \eqref{H1A''2/A22} splits.
    By the computation \eqref{imagepartial1} of $\im(\partial_1)$ and the decomposition \eqref{decompH1A''2/A21} of $H_1(\IA_n,A''_2(n)/A_{2,1}(n))$, we obtain
    \begin{gather*}
        \ker(\partial_1)\cong  
        \begin{cases}
        V_{1^2,2^21}\oplus V_{1^2,32}\oplus V_{1,21^2}\oplus V_{1,2^2}^{\oplus 2} \oplus V_{1,31}\oplus V_{0,21}^{\oplus 2} & n\ge 5\\
        V_{1^2,32}\oplus V_{1,21^2}\oplus V_{1,2^2}^{\oplus 2} \oplus V_{1,31}\oplus V_{0,21}^{\oplus 2} & n= 4\\
        V_{1,2^2} \oplus V_{1,31}\oplus V_{0,21}^{\oplus 2} & n=3.
    \end{cases}
    \end{gather*}
    By the decomposition \eqref{decompH1A21/A22} of $H_1(\IA_n,A_{2,1}(n)/A_{2,2}(n))$ and Proposition \ref{imageboundary2}, it follows that we have
    \begin{gather*}
        \Cok(\partial_2)\subset V_{0,2}.
    \end{gather*}
    Since we have a surjective $\GL(n,\Z)$-homomorphism by \eqref{H1A2''new}, it follows that
    \begin{gather*}
        \Cok(\partial_2)\cong V_{0,2},
    \end{gather*}
    which completes the proof.
\end{proof}

\section{The first homology of $\IA_n$ with coefficients in $A_{2,1}(n)$ for $n\ge 4$}\label{secH1A21}

In a way similar to the above section, we will determine the $\GL(n,\Z)$-module structure of $H_1(\IA_n,A_{2,1}(n))$ for $n\ge 4$.

The short exact sequence
\begin{gather*}
    0\to A_{2,2}(n)\to A_{2,1}(n)\to A_{2,1}(n)/A_{2,2}(n)\to 0
\end{gather*}
induces the long exact sequence 
\begin{gather}\label{H1A21}
    \begin{split}
    &\quad\quad\quad\quad\quad\cdots\quad\quad\quad\to 
    H_2(\IA_n, A_{2,1}(n))\to 
    H_2(\IA_n, A_{2,1}(n)/A_{2,2}(n))\\
    &\xrightarrow{\partial'_2} 
    H_1(\IA_n, A_{2,2}(n))\to  
    H_1(\IA_n, A_{2,1}(n))\xrightarrow{\pi_*} 
    H_1(\IA_n, A_{2,1}(n)/A_{2,2}(n))\\
    &\xrightarrow{\partial'_1}
    H_0(\IA_n, A_{2,2}(n))\to  
    H_0(\IA_n, A_{2,1}(n))\to 
    H_0(\IA_n, A_{2,1}(n)/A_{2,2}(n))\to 0.
    \end{split}
\end{gather}

Since we have
\begin{gather*}
H_0(\IA_n, A_{2,2}(n))\cong A_{2,2}(n)\cong V_{0,2},\\
H_0(\IA_n, A_{2,1}(n))\cong H_0(\IA_n, A_{2,1}(n)/A_{2,2}(n))\cong A_{2,1}(n)/A_{2,2}(n),
\end{gather*}
it follows that
\begin{gather}\label{imageboundary1'}
    \im (\partial'_1)\cong V_{0,2}.
\end{gather}

By the long exact sequence \eqref{H1A21}, we have the following short exact sequence of $\GL(n,\Z)$-representations
\begin{gather}\label{H1A21shortex}
0\to \Cok(\partial'_2)\to H_1(\IA_n,A_{2,1}(n))\xrightarrow{\pi_*} \ker(\partial'_1)\to 0.  
\end{gather}

\begin{lemma}\label{imageboundary2'}
    For $n\ge 4$, we have $\Cok(\partial'_2)=0$, that is, $\partial'_2$ is surjective.
\end{lemma}

\begin{proof}
 Set 
 \begin{gather*}
 u:=\scalebox{0.7}{$\centre{
\begingroup%
  \makeatletter%
  \providecommand\color[2][]{%
    \errmessage{(Inkscape) Color is used for the text in Inkscape, but the package 'color.sty' is not loaded}%
    \renewcommand\color[2][]{}%
  }%
  \providecommand\transparent[1]{%
    \errmessage{(Inkscape) Transparency is used (non-zero) for the text in Inkscape, but the package 'transparent.sty' is not loaded}%
    \renewcommand\transparent[1]{}%
  }%
  \providecommand\rotatebox[2]{#2}%
  \newcommand*\fsize{\dimexpr\f@size pt\relax}%
  \newcommand*\lineheight[1]{\fontsize{\fsize}{#1\fsize}\selectfont}%
  \ifx\svgwidth\undefined%
    \setlength{\unitlength}{151.47484009bp}%
    \ifx\svgscale\undefined%
      \relax%
    \else%
      \setlength{\unitlength}{\unitlength * \real{\svgscale}}%
    \fi%
  \else%
    \setlength{\unitlength}{\svgwidth}%
  \fi%
  \global\let\svgwidth\undefined%
  \global\let\svgscale\undefined%
  \makeatother%
  \begin{picture}(1,0.20576811)%
    \lineheight{1}%
    \setlength\tabcolsep{0pt}%
    \put(0,0){\includegraphics[width=\unitlength,page=1]{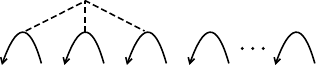}}%
  \end{picture}%
\endgroup%
}$}\in A_2(n).
 \end{gather*}
 We have
 \begin{gather*}
     \begin{split}
         \partial'_2([u]_{A_{2,2}(n)}\otimes \alpha_{0,2})
         &= [u,g_{2,1}]\otimes [g_{3,1} g_{3,2}]_{[\IA_n,\IA_n]}\\
         &=(e_1^*\cdot e_3^*)\otimes (e_1^*\otimes (e_3\wedge e_1)+e_2^*\otimes (e_3\wedge e_2))\\
         &\in H_1(\IA_n,A_{2,2}(n)).
     \end{split}
 \end{gather*}

 The image of $\partial'_2$ contains $V_{1^2,21}\oplus V_{1^2,3}\oplus V_{1,1^2}\oplus V_{0,1}$ since we have
 \begin{gather*}
     \begin{split}
         (E_{3,2}-\id)(E_{n,1}-\id)
         c_{1^2,21}\iota'\partial'_2([u]_{A_{2,2}(n)}\otimes \alpha_{0,2})&=\frac{1}{2}e_1^*\otimes (e_1^*\wedge e_2^*)\otimes (e_3\wedge e_n)\\
         (E_{3,1}-\id)(E_{n,1}-\id) c_{1^2,3}\iota'\partial'_2([u]_{A_{2,2}(n)}\otimes \alpha_{0,2})&=-(e_1^*\otimes e_1^*\otimes e_1^*)\otimes (e_3\wedge e_n)\\
         (E_{n,2}-\id) c_{1,1^2}\iota'\partial'_2([u]_{A_{2,2}(n)}\otimes \alpha_{0,2})&=\frac{1}{4}(e_1^*\wedge e_2^*)\otimes e_n\\
         c_{0,1}\iota'\partial'_2([u]_{A_{2,2}(n)}\otimes \alpha_{0,2})&=\frac{1}{2}e_1^*.
     \end{split}
 \end{gather*}

 We will use $\partial'_2 ([u]_{A_{2,2}(n)}\otimes \alpha_{2,0})$ to detect $V_{1,2}^{\oplus 2}$ in $\im(\partial'_2)$.
 We have
 \begin{gather*}
 \begin{split}
      \partial'_2([u]_{A_{2,2}(n)}\otimes \alpha_{2,0}) 
      &=[u,f_{1,2,3}]\otimes [g_{4,1}g_{4,2}g_{4,3}]_{[\IA_n,\IA_n]}\\
      &=-(e_3^* \cdot e_3^*)\otimes(e_1^*\otimes (e_4\wedge e_1)+e_2^*\otimes (e_4\wedge e_2)+e_3^*\otimes (e_4\wedge e_3)).
 \end{split} 
 \end{gather*}
 Then the image of $\partial'_2$ contains $V_{1,2}^{\oplus 2}$ since we have
 \begin{gather*}
     \begin{split}
          (E_{3,1}-\id) c_{1,2}\iota'\partial'_2([u]_{A_{2,2}(n)}\otimes \alpha_{0,2})&=\frac{1}{2}e_1^*\otimes e_1^*\otimes e_3\\
         (E_{3,1}-\id) c'_{1,2}\iota'\partial'_2([u]_{A_{2,2}(n)}\otimes \alpha_{0,2})&=- e_1^*\otimes e_1^*\otimes e_3\\
          c_{1,2}\iota'\partial'_2([u]_{A_{2,2}(n)}\otimes \alpha_{2,0})&=\frac{1}{2}e_3^*\otimes e_3^*\otimes e_4\\
          c'_{1,2}\iota'\partial'_2([u]_{A_{2,2}(n)}\otimes \alpha_{2,0})&=-\frac{3}{2}e_3^*\otimes e_3^*\otimes e_4.
     \end{split}
 \end{gather*}
 Therefore, $\partial'_2$ is surjective, and thus we have $\Cok(\partial'_2)=0$.
\end{proof}

\begin{corollary}\label{pi}
    For $n\ge 4$, the map $\pi_*$ is injective.
\end{corollary}

By combining the short exact sequence \eqref{H1A21shortex} with Lemma \ref{imageboundary2'} and with the computations \eqref{decompH1A21/A22} and \eqref{imageboundary1'}, we obtain the $\GL(n,\Z)$-module structure of $H_1(\IA_n,A_{2,1}(n))$ for $n\ge 4$.

\begin{theorem}\label{thmH1A21}
For $n\ge 4$, we have
\begin{gather*}
H_1(\IA_n,A_{2,1}(n))\cong \ker(\partial'_1)\cong
        \begin{cases}
          V_{1^2,1^4}\oplus V_{1^2,21^2}\oplus V_{1,1^3}^{\oplus 2}\oplus V_{1,21}\oplus V_{0,1^2}^{\oplus 2} & n\ge 6\\
          V_{1^2,21^2}\oplus V_{1,1^3}^{\oplus 2}\oplus V_{1,21}\oplus V_{0,1^2}^{\oplus 2} & n=5\\
          V_{1,1^3}\oplus V_{1,21}\oplus V_{0,1^2}^{\oplus 2} & n=4.
        \end{cases}
\end{gather*}
\end{theorem}

\section{The first homology of $\IA_n$ with coefficients in $A''_2(n)$ for $n\ge 4$}\label{secH1A''2}

Here we will determine the $\GL(n,\Z)$-module structure of $H_1(\IA_n,A''_2(n))$ for $n\ge 4$. 

\begin{theorem}\label{thmH1A2''}
For $n\ge 4$, we have
\begin{gather*}
  \begin{split}
        H_1(\IA_n, A''_2(n))\cong 
        \begin{cases}
        V_{1^2,2^21}\oplus V_{1^2,32}\oplus V_{1,21^2}\oplus V_{1,2^2}^{\oplus 2} \oplus V_{1,31}\oplus V_{0,21}^{\oplus 2} & n\ge 5\\
        V_{1^2,32}\oplus V_{1,21^2}\oplus V_{1,2^2}^{\oplus 2} \oplus V_{1,31}\oplus V_{0,21}^{\oplus 2} & n= 4.
        \end{cases}
  \end{split} 
\end{gather*}
\end{theorem}

\begin{proof}
The short exact sequence
\begin{gather*}
    0\to A_{2,1}(n)\to A''_2(n)\to A''_2(n)/A_{2,1}(n)\to 0
\end{gather*}
induces a long exact sequence
\begin{gather}\label{H1A2''full}
\begin{split}
    &\quad\quad\quad\quad\quad\cdots\quad\quad\;\;\;\;\to 
    H_2(\IA_n, A''_2(n))\to 
    H_2(\IA_n, A''_2(n)/A_{2,1}(n))\\
    &\xrightarrow{\partial''_2} 
    H_1(\IA_n, A_{2,1}(n))\to  
    H_1(\IA_n, A''_2(n))\to 
    H_1(\IA_n, A''_2(n)/A_{2,1}(n))\\
    &\xrightarrow{\partial''_1}
    H_0(\IA_n, A_{2,1}(n))\to  
    H_0(\IA_n, A''_2(n))\to 
    H_0(\IA_n, A''_2(n)/A_{2,1}(n))\to 0.
\end{split}
\end{gather}
Since we have
\begin{gather*}
        H_0((\IA_n, A_{2,1}(n))\cong A_{2,1}(n)/A_{2,2}(n)\cong V_{0,1^3},\\
        H_0(\IA_n, A''_2(n))\cong H_0(\IA_n, A''_2(n)/A_{2,1}(n))\cong A''_2(n)/A_{2,1}(n),
\end{gather*}
it follows that
\begin{gather}\label{imageboundary''1}
    \im (\partial''_1)\cong V_{0,1^3}.
\end{gather}

By the long exact sequence \eqref{H1A2''full}, we have the following short exact sequence of $\GL(n,\Z)$-representations
\begin{gather}\label{H1A2''shortex}
0\to \Cok(\partial''_2)\to H_1(\IA_n,A''_2(n))\to \ker(\partial''_1)\to 0.  
\end{gather}

Note that we have the following commutative diagram
\begin{gather}\label{commutativediagrampartial2}
    \xymatrix{
    H_2(\IA_n, A''_2(n)/A_{2,1}(n))\ar[rr]^{\partial_2}\ar[rd]^{\partial''_2}&& H_1(\IA_n,A_{2,1}(n)/A_{2,2}(n))
    \\
    &H_1(\IA_n, A_{2,1}(n))\ar[ru]^{\pi_*}}
\end{gather}
where $\partial_2$ appeared in \eqref{H1A2''/A22} and $\pi_*$ appeared in \eqref{H1A21}.
Therefore, we have
\begin{gather*}
    \im (\partial_2)=\im (\pi_* \partial''_2).
\end{gather*}
Since $\pi_*$ is injective by Corollary \ref{pi}, we have 
\begin{gather*}
    \im (\pi_* \partial''_2)\cong \im (\partial''_2).
\end{gather*}
Hence, by Proposition \ref{imageboundary2} and Theorem \ref{thmH1A21}, the boundary map $\partial''_2$ is surjective for $n\ge 4$.
It follows from \eqref{decompH1A''2/A21}, \eqref{H1A2''shortex} and \eqref{imageboundary''1} that 
\begin{gather*}
  \begin{split}
        H_1(\IA_n, A''_2(n))&\cong \ker (\partial''_1)\\
        &\cong 
        \begin{cases}
        V_{1^2,2^21}\oplus V_{1^2,32}\oplus V_{1,21^2}\oplus V_{1,2^2}^{\oplus 2} \oplus V_{1,31}\oplus V_{0,21}^{\oplus 2} & n\ge 5\\
        V_{1^2,32}\oplus V_{1,21^2}\oplus V_{1,2^2}^{\oplus 2} \oplus V_{1,31}\oplus V_{0,21}^{\oplus 2} & n= 4.
        \end{cases}
  \end{split} 
\end{gather*}    
\end{proof}

\section{Self duality of $A''_2(3)$}\label{secselfdual}

In this section, we will prove that $A''_2(3)$ is a self-dual $\Aut(F_3)$-module, which appeared in our previous paper \cite[Remark 7.13]{Katada1} as a conjecture.

\subsection{Basis for $A''_2(3)$}\label{subsecbasis}

Since $\dim V_{0,2^2}(3)=6$, $\dim V_{0,1^3}(3)=1$ and $\dim V_{0,2}(3)=6$, it follows that 
$\dim A''_2(3)=13$.
Moreover, we can take the following basis for $A''_2(3)$:
\begin{gather*}
    \begin{split}
        v_1&=\scalebox{0.8}{$\centre{
\begingroup%
  \makeatletter%
  \providecommand\color[2][]{%
    \errmessage{(Inkscape) Color is used for the text in Inkscape, but the package 'color.sty' is not loaded}%
    \renewcommand\color[2][]{}%
  }%
  \providecommand\transparent[1]{%
    \errmessage{(Inkscape) Transparency is used (non-zero) for the text in Inkscape, but the package 'transparent.sty' is not loaded}%
    \renewcommand\transparent[1]{}%
  }%
  \providecommand\rotatebox[2]{#2}%
  \newcommand*\fsize{\dimexpr\f@size pt\relax}%
  \newcommand*\lineheight[1]{\fontsize{\fsize}{#1\fsize}\selectfont}%
  \ifx\svgwidth\undefined%
    \setlength{\unitlength}{80.22484197bp}%
    \ifx\svgscale\undefined%
      \relax%
    \else%
      \setlength{\unitlength}{\unitlength * \real{\svgscale}}%
    \fi%
  \else%
    \setlength{\unitlength}{\svgwidth}%
  \fi%
  \global\let\svgwidth\undefined%
  \global\let\svgscale\undefined%
  \makeatother%
  \begin{picture}(1,0.34317237)%
    \lineheight{1}%
    \setlength\tabcolsep{0pt}%
    \put(0,0){\includegraphics[width=\unitlength,page=1]{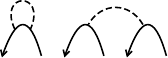}}%
  \end{picture}%
\endgroup%
}$}-\frac{1}{2}\scalebox{0.8}{$\centre{
\begingroup%
  \makeatletter%
  \providecommand\color[2][]{%
    \errmessage{(Inkscape) Color is used for the text in Inkscape, but the package 'color.sty' is not loaded}%
    \renewcommand\color[2][]{}%
  }%
  \providecommand\transparent[1]{%
    \errmessage{(Inkscape) Transparency is used (non-zero) for the text in Inkscape, but the package 'transparent.sty' is not loaded}%
    \renewcommand\transparent[1]{}%
  }%
  \providecommand\rotatebox[2]{#2}%
  \newcommand*\fsize{\dimexpr\f@size pt\relax}%
  \newcommand*\lineheight[1]{\fontsize{\fsize}{#1\fsize}\selectfont}%
  \ifx\svgwidth\undefined%
    \setlength{\unitlength}{80.22484197bp}%
    \ifx\svgscale\undefined%
      \relax%
    \else%
      \setlength{\unitlength}{\unitlength * \real{\svgscale}}%
    \fi%
  \else%
    \setlength{\unitlength}{\svgwidth}%
  \fi%
  \global\let\svgwidth\undefined%
  \global\let\svgscale\undefined%
  \makeatother%
  \begin{picture}(1,0.37608811)%
    \lineheight{1}%
    \setlength\tabcolsep{0pt}%
    \put(0,0){\includegraphics[width=\unitlength,page=1]{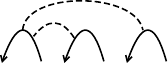}}%
  \end{picture}%
\endgroup%
}$}-\frac{1}{2}\scalebox{0.8}{$\centre{
\begingroup%
  \makeatletter%
  \providecommand\color[2][]{%
    \errmessage{(Inkscape) Color is used for the text in Inkscape, but the package 'color.sty' is not loaded}%
    \renewcommand\color[2][]{}%
  }%
  \providecommand\transparent[1]{%
    \errmessage{(Inkscape) Transparency is used (non-zero) for the text in Inkscape, but the package 'transparent.sty' is not loaded}%
    \renewcommand\transparent[1]{}%
  }%
  \providecommand\rotatebox[2]{#2}%
  \newcommand*\fsize{\dimexpr\f@size pt\relax}%
  \newcommand*\lineheight[1]{\fontsize{\fsize}{#1\fsize}\selectfont}%
  \ifx\svgwidth\undefined%
    \setlength{\unitlength}{80.22484197bp}%
    \ifx\svgscale\undefined%
      \relax%
    \else%
      \setlength{\unitlength}{\unitlength * \real{\svgscale}}%
    \fi%
  \else%
    \setlength{\unitlength}{\svgwidth}%
  \fi%
  \global\let\svgwidth\undefined%
  \global\let\svgscale\undefined%
  \makeatother%
  \begin{picture}(1,0.35443458)%
    \lineheight{1}%
    \setlength\tabcolsep{0pt}%
    \put(0,0){\includegraphics[width=\unitlength,page=1]{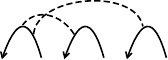}}%
  \end{picture}%
\endgroup%
}$},\\
        v_2&=\scalebox{0.8}{$\centre{
\begingroup%
  \makeatletter%
  \providecommand\color[2][]{%
    \errmessage{(Inkscape) Color is used for the text in Inkscape, but the package 'color.sty' is not loaded}%
    \renewcommand\color[2][]{}%
  }%
  \providecommand\transparent[1]{%
    \errmessage{(Inkscape) Transparency is used (non-zero) for the text in Inkscape, but the package 'transparent.sty' is not loaded}%
    \renewcommand\transparent[1]{}%
  }%
  \providecommand\rotatebox[2]{#2}%
  \newcommand*\fsize{\dimexpr\f@size pt\relax}%
  \newcommand*\lineheight[1]{\fontsize{\fsize}{#1\fsize}\selectfont}%
  \ifx\svgwidth\undefined%
    \setlength{\unitlength}{80.22484197bp}%
    \ifx\svgscale\undefined%
      \relax%
    \else%
      \setlength{\unitlength}{\unitlength * \real{\svgscale}}%
    \fi%
  \else%
    \setlength{\unitlength}{\svgwidth}%
  \fi%
  \global\let\svgwidth\undefined%
  \global\let\svgscale\undefined%
  \makeatother%
  \begin{picture}(1,0.37750979)%
    \lineheight{1}%
    \setlength\tabcolsep{0pt}%
    \put(0,0){\includegraphics[width=\unitlength,page=1]{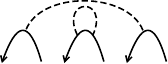}}%
  \end{picture}%
\endgroup%
}$}-\frac{1}{2}\scalebox{0.8}{$\centre{
\begingroup%
  \makeatletter%
  \providecommand\color[2][]{%
    \errmessage{(Inkscape) Color is used for the text in Inkscape, but the package 'color.sty' is not loaded}%
    \renewcommand\color[2][]{}%
  }%
  \providecommand\transparent[1]{%
    \errmessage{(Inkscape) Transparency is used (non-zero) for the text in Inkscape, but the package 'transparent.sty' is not loaded}%
    \renewcommand\transparent[1]{}%
  }%
  \providecommand\rotatebox[2]{#2}%
  \newcommand*\fsize{\dimexpr\f@size pt\relax}%
  \newcommand*\lineheight[1]{\fontsize{\fsize}{#1\fsize}\selectfont}%
  \ifx\svgwidth\undefined%
    \setlength{\unitlength}{80.22484197bp}%
    \ifx\svgscale\undefined%
      \relax%
    \else%
      \setlength{\unitlength}{\unitlength * \real{\svgscale}}%
    \fi%
  \else%
    \setlength{\unitlength}{\svgwidth}%
  \fi%
  \global\let\svgwidth\undefined%
  \global\let\svgscale\undefined%
  \makeatother%
  \begin{picture}(1,0.28950132)%
    \lineheight{1}%
    \setlength\tabcolsep{0pt}%
    \put(0,0){\includegraphics[width=\unitlength,page=1]{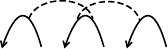}}%
  \end{picture}%
\endgroup%
}$}-\frac{1}{2}\scalebox{0.8}{$\centre{
\begingroup%
  \makeatletter%
  \providecommand\color[2][]{%
    \errmessage{(Inkscape) Color is used for the text in Inkscape, but the package 'color.sty' is not loaded}%
    \renewcommand\color[2][]{}%
  }%
  \providecommand\transparent[1]{%
    \errmessage{(Inkscape) Transparency is used (non-zero) for the text in Inkscape, but the package 'transparent.sty' is not loaded}%
    \renewcommand\transparent[1]{}%
  }%
  \providecommand\rotatebox[2]{#2}%
  \newcommand*\fsize{\dimexpr\f@size pt\relax}%
  \newcommand*\lineheight[1]{\fontsize{\fsize}{#1\fsize}\selectfont}%
  \ifx\svgwidth\undefined%
    \setlength{\unitlength}{80.22484197bp}%
    \ifx\svgscale\undefined%
      \relax%
    \else%
      \setlength{\unitlength}{\unitlength * \real{\svgscale}}%
    \fi%
  \else%
    \setlength{\unitlength}{\svgwidth}%
  \fi%
  \global\let\svgwidth\undefined%
  \global\let\svgscale\undefined%
  \makeatother%
  \begin{picture}(1,0.28487196)%
    \lineheight{1}%
    \setlength\tabcolsep{0pt}%
    \put(0,0){\includegraphics[width=\unitlength,page=1]{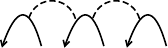}}%
  \end{picture}%
\endgroup%
}$},\\
        v_3&=\scalebox{0.8}{$\centre{
\begingroup%
  \makeatletter%
  \providecommand\color[2][]{%
    \errmessage{(Inkscape) Color is used for the text in Inkscape, but the package 'color.sty' is not loaded}%
    \renewcommand\color[2][]{}%
  }%
  \providecommand\transparent[1]{%
    \errmessage{(Inkscape) Transparency is used (non-zero) for the text in Inkscape, but the package 'transparent.sty' is not loaded}%
    \renewcommand\transparent[1]{}%
  }%
  \providecommand\rotatebox[2]{#2}%
  \newcommand*\fsize{\dimexpr\f@size pt\relax}%
  \newcommand*\lineheight[1]{\fontsize{\fsize}{#1\fsize}\selectfont}%
  \ifx\svgwidth\undefined%
    \setlength{\unitlength}{80.22484197bp}%
    \ifx\svgscale\undefined%
      \relax%
    \else%
      \setlength{\unitlength}{\unitlength * \real{\svgscale}}%
    \fi%
  \else%
    \setlength{\unitlength}{\svgwidth}%
  \fi%
  \global\let\svgwidth\undefined%
  \global\let\svgscale\undefined%
  \makeatother%
  \begin{picture}(1,0.37291581)%
    \lineheight{1}%
    \setlength\tabcolsep{0pt}%
    \put(0,0){\includegraphics[width=\unitlength,page=1]{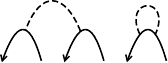}}%
  \end{picture}%
\endgroup%
}$}-\frac{1}{2}\scalebox{0.8}{$\centre{
\begingroup%
  \makeatletter%
  \providecommand\color[2][]{%
    \errmessage{(Inkscape) Color is used for the text in Inkscape, but the package 'color.sty' is not loaded}%
    \renewcommand\color[2][]{}%
  }%
  \providecommand\transparent[1]{%
    \errmessage{(Inkscape) Transparency is used (non-zero) for the text in Inkscape, but the package 'transparent.sty' is not loaded}%
    \renewcommand\transparent[1]{}%
  }%
  \providecommand\rotatebox[2]{#2}%
  \newcommand*\fsize{\dimexpr\f@size pt\relax}%
  \newcommand*\lineheight[1]{\fontsize{\fsize}{#1\fsize}\selectfont}%
  \ifx\svgwidth\undefined%
    \setlength{\unitlength}{80.22484197bp}%
    \ifx\svgscale\undefined%
      \relax%
    \else%
      \setlength{\unitlength}{\unitlength * \real{\svgscale}}%
    \fi%
  \else%
    \setlength{\unitlength}{\svgwidth}%
  \fi%
  \global\let\svgwidth\undefined%
  \global\let\svgscale\undefined%
  \makeatother%
  \begin{picture}(1,0.33636857)%
    \lineheight{1}%
    \setlength\tabcolsep{0pt}%
    \put(0,0){\includegraphics[width=\unitlength,page=1]{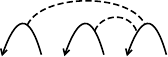}}%
  \end{picture}%
\endgroup%
}$}-\frac{1}{2}\scalebox{0.8}{$\centre{
\begingroup%
  \makeatletter%
  \providecommand\color[2][]{%
    \errmessage{(Inkscape) Color is used for the text in Inkscape, but the package 'color.sty' is not loaded}%
    \renewcommand\color[2][]{}%
  }%
  \providecommand\transparent[1]{%
    \errmessage{(Inkscape) Transparency is used (non-zero) for the text in Inkscape, but the package 'transparent.sty' is not loaded}%
    \renewcommand\transparent[1]{}%
  }%
  \providecommand\rotatebox[2]{#2}%
  \newcommand*\fsize{\dimexpr\f@size pt\relax}%
  \newcommand*\lineheight[1]{\fontsize{\fsize}{#1\fsize}\selectfont}%
  \ifx\svgwidth\undefined%
    \setlength{\unitlength}{80.22484197bp}%
    \ifx\svgscale\undefined%
      \relax%
    \else%
      \setlength{\unitlength}{\unitlength * \real{\svgscale}}%
    \fi%
  \else%
    \setlength{\unitlength}{\svgwidth}%
  \fi%
  \global\let\svgwidth\undefined%
  \global\let\svgscale\undefined%
  \makeatother%
  \begin{picture}(1,0.3142311)%
    \lineheight{1}%
    \setlength\tabcolsep{0pt}%
    \put(0,0){\includegraphics[width=\unitlength,page=1]{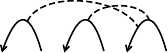}}%
  \end{picture}%
\endgroup%
}$},\\
        v_{12}&=\scalebox{0.8}{$\centre{
\begingroup%
  \makeatletter%
  \providecommand\color[2][]{%
    \errmessage{(Inkscape) Color is used for the text in Inkscape, but the package 'color.sty' is not loaded}%
    \renewcommand\color[2][]{}%
  }%
  \providecommand\transparent[1]{%
    \errmessage{(Inkscape) Transparency is used (non-zero) for the text in Inkscape, but the package 'transparent.sty' is not loaded}%
    \renewcommand\transparent[1]{}%
  }%
  \providecommand\rotatebox[2]{#2}%
  \newcommand*\fsize{\dimexpr\f@size pt\relax}%
  \newcommand*\lineheight[1]{\fontsize{\fsize}{#1\fsize}\selectfont}%
  \ifx\svgwidth\undefined%
    \setlength{\unitlength}{80.22484197bp}%
    \ifx\svgscale\undefined%
      \relax%
    \else%
      \setlength{\unitlength}{\unitlength * \real{\svgscale}}%
    \fi%
  \else%
    \setlength{\unitlength}{\svgwidth}%
  \fi%
  \global\let\svgwidth\undefined%
  \global\let\svgscale\undefined%
  \makeatother%
  \begin{picture}(1,0.34317238)%
    \lineheight{1}%
    \setlength\tabcolsep{0pt}%
    \put(0,0){\includegraphics[width=\unitlength,page=1]{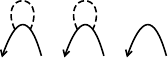}}%
  \end{picture}%
\endgroup%
}$}-\frac{1}{2}\scalebox{0.8}{$\centre{
\begingroup%
  \makeatletter%
  \providecommand\color[2][]{%
    \errmessage{(Inkscape) Color is used for the text in Inkscape, but the package 'color.sty' is not loaded}%
    \renewcommand\color[2][]{}%
  }%
  \providecommand\transparent[1]{%
    \errmessage{(Inkscape) Transparency is used (non-zero) for the text in Inkscape, but the package 'transparent.sty' is not loaded}%
    \renewcommand\transparent[1]{}%
  }%
  \providecommand\rotatebox[2]{#2}%
  \newcommand*\fsize{\dimexpr\f@size pt\relax}%
  \newcommand*\lineheight[1]{\fontsize{\fsize}{#1\fsize}\selectfont}%
  \ifx\svgwidth\undefined%
    \setlength{\unitlength}{80.22484197bp}%
    \ifx\svgscale\undefined%
      \relax%
    \else%
      \setlength{\unitlength}{\unitlength * \real{\svgscale}}%
    \fi%
  \else%
    \setlength{\unitlength}{\svgwidth}%
  \fi%
  \global\let\svgwidth\undefined%
  \global\let\svgscale\undefined%
  \makeatother%
  \begin{picture}(1,0.32120476)%
    \lineheight{1}%
    \setlength\tabcolsep{0pt}%
    \put(0,0){\includegraphics[width=\unitlength,page=1]{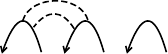}}%
  \end{picture}%
\endgroup%
}$}-\frac{1}{2}\scalebox{0.8}{$\centre{
\begingroup%
  \makeatletter%
  \providecommand\color[2][]{%
    \errmessage{(Inkscape) Color is used for the text in Inkscape, but the package 'color.sty' is not loaded}%
    \renewcommand\color[2][]{}%
  }%
  \providecommand\transparent[1]{%
    \errmessage{(Inkscape) Transparency is used (non-zero) for the text in Inkscape, but the package 'transparent.sty' is not loaded}%
    \renewcommand\transparent[1]{}%
  }%
  \providecommand\rotatebox[2]{#2}%
  \newcommand*\fsize{\dimexpr\f@size pt\relax}%
  \newcommand*\lineheight[1]{\fontsize{\fsize}{#1\fsize}\selectfont}%
  \ifx\svgwidth\undefined%
    \setlength{\unitlength}{80.22484197bp}%
    \ifx\svgscale\undefined%
      \relax%
    \else%
      \setlength{\unitlength}{\unitlength * \real{\svgscale}}%
    \fi%
  \else%
    \setlength{\unitlength}{\svgwidth}%
  \fi%
  \global\let\svgwidth\undefined%
  \global\let\svgscale\undefined%
  \makeatother%
  \begin{picture}(1,0.27113311)%
    \lineheight{1}%
    \setlength\tabcolsep{0pt}%
    \put(0,0){\includegraphics[width=\unitlength,page=1]{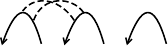}}%
  \end{picture}%
\endgroup%
}$},\\
        v_{13}&=\scalebox{0.8}{$\centre{
\begingroup%
  \makeatletter%
  \providecommand\color[2][]{%
    \errmessage{(Inkscape) Color is used for the text in Inkscape, but the package 'color.sty' is not loaded}%
    \renewcommand\color[2][]{}%
  }%
  \providecommand\transparent[1]{%
    \errmessage{(Inkscape) Transparency is used (non-zero) for the text in Inkscape, but the package 'transparent.sty' is not loaded}%
    \renewcommand\transparent[1]{}%
  }%
  \providecommand\rotatebox[2]{#2}%
  \newcommand*\fsize{\dimexpr\f@size pt\relax}%
  \newcommand*\lineheight[1]{\fontsize{\fsize}{#1\fsize}\selectfont}%
  \ifx\svgwidth\undefined%
    \setlength{\unitlength}{80.22484197bp}%
    \ifx\svgscale\undefined%
      \relax%
    \else%
      \setlength{\unitlength}{\unitlength * \real{\svgscale}}%
    \fi%
  \else%
    \setlength{\unitlength}{\svgwidth}%
  \fi%
  \global\let\svgwidth\undefined%
  \global\let\svgscale\undefined%
  \makeatother%
  \begin{picture}(1,0.34317239)%
    \lineheight{1}%
    \setlength\tabcolsep{0pt}%
    \put(0,0){\includegraphics[width=\unitlength,page=1]{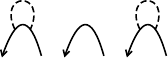}}%
  \end{picture}%
\endgroup%
}$}-\frac{1}{2}\scalebox{0.8}{$\centre{
\begingroup%
  \makeatletter%
  \providecommand\color[2][]{%
    \errmessage{(Inkscape) Color is used for the text in Inkscape, but the package 'color.sty' is not loaded}%
    \renewcommand\color[2][]{}%
  }%
  \providecommand\transparent[1]{%
    \errmessage{(Inkscape) Transparency is used (non-zero) for the text in Inkscape, but the package 'transparent.sty' is not loaded}%
    \renewcommand\transparent[1]{}%
  }%
  \providecommand\rotatebox[2]{#2}%
  \newcommand*\fsize{\dimexpr\f@size pt\relax}%
  \newcommand*\lineheight[1]{\fontsize{\fsize}{#1\fsize}\selectfont}%
  \ifx\svgwidth\undefined%
    \setlength{\unitlength}{80.22484197bp}%
    \ifx\svgscale\undefined%
      \relax%
    \else%
      \setlength{\unitlength}{\unitlength * \real{\svgscale}}%
    \fi%
  \else%
    \setlength{\unitlength}{\svgwidth}%
  \fi%
  \global\let\svgwidth\undefined%
  \global\let\svgscale\undefined%
  \makeatother%
  \begin{picture}(1,0.37608815)%
    \lineheight{1}%
    \setlength\tabcolsep{0pt}%
    \put(0,0){\includegraphics[width=\unitlength,page=1]{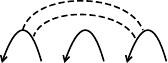}}%
  \end{picture}%
\endgroup%
}$}-\frac{1}{2}\scalebox{0.8}{$\centre{
\begingroup%
  \makeatletter%
  \providecommand\color[2][]{%
    \errmessage{(Inkscape) Color is used for the text in Inkscape, but the package 'color.sty' is not loaded}%
    \renewcommand\color[2][]{}%
  }%
  \providecommand\transparent[1]{%
    \errmessage{(Inkscape) Transparency is used (non-zero) for the text in Inkscape, but the package 'transparent.sty' is not loaded}%
    \renewcommand\transparent[1]{}%
  }%
  \providecommand\rotatebox[2]{#2}%
  \newcommand*\fsize{\dimexpr\f@size pt\relax}%
  \newcommand*\lineheight[1]{\fontsize{\fsize}{#1\fsize}\selectfont}%
  \ifx\svgwidth\undefined%
    \setlength{\unitlength}{80.22484197bp}%
    \ifx\svgscale\undefined%
      \relax%
    \else%
      \setlength{\unitlength}{\unitlength * \real{\svgscale}}%
    \fi%
  \else%
    \setlength{\unitlength}{\svgwidth}%
  \fi%
  \global\let\svgwidth\undefined%
  \global\let\svgscale\undefined%
  \makeatother%
  \begin{picture}(1,0.29934562)%
    \lineheight{1}%
    \setlength\tabcolsep{0pt}%
    \put(0,0){\includegraphics[width=\unitlength,page=1]{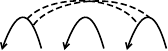}}%
  \end{picture}%
\endgroup%
}$},\\
        v_{23}&=\scalebox{0.8}{$\centre{
\begingroup%
  \makeatletter%
  \providecommand\color[2][]{%
    \errmessage{(Inkscape) Color is used for the text in Inkscape, but the package 'color.sty' is not loaded}%
    \renewcommand\color[2][]{}%
  }%
  \providecommand\transparent[1]{%
    \errmessage{(Inkscape) Transparency is used (non-zero) for the text in Inkscape, but the package 'transparent.sty' is not loaded}%
    \renewcommand\transparent[1]{}%
  }%
  \providecommand\rotatebox[2]{#2}%
  \newcommand*\fsize{\dimexpr\f@size pt\relax}%
  \newcommand*\lineheight[1]{\fontsize{\fsize}{#1\fsize}\selectfont}%
  \ifx\svgwidth\undefined%
    \setlength{\unitlength}{80.22484197bp}%
    \ifx\svgscale\undefined%
      \relax%
    \else%
      \setlength{\unitlength}{\unitlength * \real{\svgscale}}%
    \fi%
  \else%
    \setlength{\unitlength}{\svgwidth}%
  \fi%
  \global\let\svgwidth\undefined%
  \global\let\svgscale\undefined%
  \makeatother%
  \begin{picture}(1,0.34317239)%
    \lineheight{1}%
    \setlength\tabcolsep{0pt}%
    \put(0,0){\includegraphics[width=\unitlength,page=1]{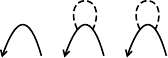}}%
  \end{picture}%
\endgroup%
}$}-\frac{1}{2}\scalebox{0.8}{$\centre{
\begingroup%
  \makeatletter%
  \providecommand\color[2][]{%
    \errmessage{(Inkscape) Color is used for the text in Inkscape, but the package 'color.sty' is not loaded}%
    \renewcommand\color[2][]{}%
  }%
  \providecommand\transparent[1]{%
    \errmessage{(Inkscape) Transparency is used (non-zero) for the text in Inkscape, but the package 'transparent.sty' is not loaded}%
    \renewcommand\transparent[1]{}%
  }%
  \providecommand\rotatebox[2]{#2}%
  \newcommand*\fsize{\dimexpr\f@size pt\relax}%
  \newcommand*\lineheight[1]{\fontsize{\fsize}{#1\fsize}\selectfont}%
  \ifx\svgwidth\undefined%
    \setlength{\unitlength}{80.22484197bp}%
    \ifx\svgscale\undefined%
      \relax%
    \else%
      \setlength{\unitlength}{\unitlength * \real{\svgscale}}%
    \fi%
  \else%
    \setlength{\unitlength}{\svgwidth}%
  \fi%
  \global\let\svgwidth\undefined%
  \global\let\svgscale\undefined%
  \makeatother%
  \begin{picture}(1,0.29428627)%
    \lineheight{1}%
    \setlength\tabcolsep{0pt}%
    \put(0,0){\includegraphics[width=\unitlength,page=1]{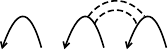}}%
  \end{picture}%
\endgroup%
}$}-\frac{1}{2}\scalebox{0.8}{$\centre{
\begingroup%
  \makeatletter%
  \providecommand\color[2][]{%
    \errmessage{(Inkscape) Color is used for the text in Inkscape, but the package 'color.sty' is not loaded}%
    \renewcommand\color[2][]{}%
  }%
  \providecommand\transparent[1]{%
    \errmessage{(Inkscape) Transparency is used (non-zero) for the text in Inkscape, but the package 'transparent.sty' is not loaded}%
    \renewcommand\transparent[1]{}%
  }%
  \providecommand\rotatebox[2]{#2}%
  \newcommand*\fsize{\dimexpr\f@size pt\relax}%
  \newcommand*\lineheight[1]{\fontsize{\fsize}{#1\fsize}\selectfont}%
  \ifx\svgwidth\undefined%
    \setlength{\unitlength}{80.22484197bp}%
    \ifx\svgscale\undefined%
      \relax%
    \else%
      \setlength{\unitlength}{\unitlength * \real{\svgscale}}%
    \fi%
  \else%
    \setlength{\unitlength}{\svgwidth}%
  \fi%
  \global\let\svgwidth\undefined%
  \global\let\svgscale\undefined%
  \makeatother%
  \begin{picture}(1,0.26850143)%
    \lineheight{1}%
    \setlength\tabcolsep{0pt}%
    \put(0,0){\includegraphics[width=\unitlength,page=1]{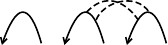}}%
  \end{picture}%
\endgroup%
}$},\\
        u_{123}&=\scalebox{0.8}{$\centre{}$},\\
        w_1&=\scalebox{0.8}{$\centre{
\begingroup%
  \makeatletter%
  \providecommand\color[2][]{%
    \errmessage{(Inkscape) Color is used for the text in Inkscape, but the package 'color.sty' is not loaded}%
    \renewcommand\color[2][]{}%
  }%
  \providecommand\transparent[1]{%
    \errmessage{(Inkscape) Transparency is used (non-zero) for the text in Inkscape, but the package 'transparent.sty' is not loaded}%
    \renewcommand\transparent[1]{}%
  }%
  \providecommand\rotatebox[2]{#2}%
  \newcommand*\fsize{\dimexpr\f@size pt\relax}%
  \newcommand*\lineheight[1]{\fontsize{\fsize}{#1\fsize}\selectfont}%
  \ifx\svgwidth\undefined%
    \setlength{\unitlength}{72.72484094bp}%
    \ifx\svgscale\undefined%
      \relax%
    \else%
      \setlength{\unitlength}{\unitlength * \real{\svgscale}}%
    \fi%
  \else%
    \setlength{\unitlength}{\svgwidth}%
  \fi%
  \global\let\svgwidth\undefined%
  \global\let\svgscale\undefined%
  \makeatother%
  \begin{picture}(1,0.47368456)%
    \lineheight{1}%
    \setlength\tabcolsep{0pt}%
    \put(0,0){\includegraphics[width=\unitlength,page=1]{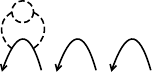}}%
  \end{picture}%
\endgroup%
}$},\quad
        w_2=\scalebox{0.8}{$\centre{
\begingroup%
  \makeatletter%
  \providecommand\color[2][]{%
    \errmessage{(Inkscape) Color is used for the text in Inkscape, but the package 'color.sty' is not loaded}%
    \renewcommand\color[2][]{}%
  }%
  \providecommand\transparent[1]{%
    \errmessage{(Inkscape) Transparency is used (non-zero) for the text in Inkscape, but the package 'transparent.sty' is not loaded}%
    \renewcommand\transparent[1]{}%
  }%
  \providecommand\rotatebox[2]{#2}%
  \newcommand*\fsize{\dimexpr\f@size pt\relax}%
  \newcommand*\lineheight[1]{\fontsize{\fsize}{#1\fsize}\selectfont}%
  \ifx\svgwidth\undefined%
    \setlength{\unitlength}{72.72484094bp}%
    \ifx\svgscale\undefined%
      \relax%
    \else%
      \setlength{\unitlength}{\unitlength * \real{\svgscale}}%
    \fi%
  \else%
    \setlength{\unitlength}{\svgwidth}%
  \fi%
  \global\let\svgwidth\undefined%
  \global\let\svgscale\undefined%
  \makeatother%
  \begin{picture}(1,0.47368456)%
    \lineheight{1}%
    \setlength\tabcolsep{0pt}%
    \put(0,0){\includegraphics[width=\unitlength,page=1]{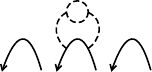}}%
  \end{picture}%
\endgroup%
}$},\quad
        w_3=\scalebox{0.8}{$\centre{
\begingroup%
  \makeatletter%
  \providecommand\color[2][]{%
    \errmessage{(Inkscape) Color is used for the text in Inkscape, but the package 'color.sty' is not loaded}%
    \renewcommand\color[2][]{}%
  }%
  \providecommand\transparent[1]{%
    \errmessage{(Inkscape) Transparency is used (non-zero) for the text in Inkscape, but the package 'transparent.sty' is not loaded}%
    \renewcommand\transparent[1]{}%
  }%
  \providecommand\rotatebox[2]{#2}%
  \newcommand*\fsize{\dimexpr\f@size pt\relax}%
  \newcommand*\lineheight[1]{\fontsize{\fsize}{#1\fsize}\selectfont}%
  \ifx\svgwidth\undefined%
    \setlength{\unitlength}{74.26292919bp}%
    \ifx\svgscale\undefined%
      \relax%
    \else%
      \setlength{\unitlength}{\unitlength * \real{\svgscale}}%
    \fi%
  \else%
    \setlength{\unitlength}{\svgwidth}%
  \fi%
  \global\let\svgwidth\undefined%
  \global\let\svgscale\undefined%
  \makeatother%
  \begin{picture}(1,0.46387389)%
    \lineheight{1}%
    \setlength\tabcolsep{0pt}%
    \put(0,0){\includegraphics[width=\unitlength,page=1]{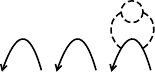}}%
  \end{picture}%
\endgroup%
}$},\\
        w_{12}&=\scalebox{0.8}{$\centre{
\begingroup%
  \makeatletter%
  \providecommand\color[2][]{%
    \errmessage{(Inkscape) Color is used for the text in Inkscape, but the package 'color.sty' is not loaded}%
    \renewcommand\color[2][]{}%
  }%
  \providecommand\transparent[1]{%
    \errmessage{(Inkscape) Transparency is used (non-zero) for the text in Inkscape, but the package 'transparent.sty' is not loaded}%
    \renewcommand\transparent[1]{}%
  }%
  \providecommand\rotatebox[2]{#2}%
  \newcommand*\fsize{\dimexpr\f@size pt\relax}%
  \newcommand*\lineheight[1]{\fontsize{\fsize}{#1\fsize}\selectfont}%
  \ifx\svgwidth\undefined%
    \setlength{\unitlength}{80.22484197bp}%
    \ifx\svgscale\undefined%
      \relax%
    \else%
      \setlength{\unitlength}{\unitlength * \real{\svgscale}}%
    \fi%
  \else%
    \setlength{\unitlength}{\svgwidth}%
  \fi%
  \global\let\svgwidth\undefined%
  \global\let\svgscale\undefined%
  \makeatother%
  \begin{picture}(1,0.45321485)%
    \lineheight{1}%
    \setlength\tabcolsep{0pt}%
    \put(0,0){\includegraphics[width=\unitlength,page=1]{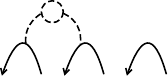}}%
  \end{picture}%
\endgroup%
}$},\quad
        w_{13}=\scalebox{0.8}{$\centre{}$},\quad
        w_{23}=\scalebox{0.8}{$\centre{
\begingroup%
  \makeatletter%
  \providecommand\color[2][]{%
    \errmessage{(Inkscape) Color is used for the text in Inkscape, but the package 'color.sty' is not loaded}%
    \renewcommand\color[2][]{}%
  }%
  \providecommand\transparent[1]{%
    \errmessage{(Inkscape) Transparency is used (non-zero) for the text in Inkscape, but the package 'transparent.sty' is not loaded}%
    \renewcommand\transparent[1]{}%
  }%
  \providecommand\rotatebox[2]{#2}%
  \newcommand*\fsize{\dimexpr\f@size pt\relax}%
  \newcommand*\lineheight[1]{\fontsize{\fsize}{#1\fsize}\selectfont}%
  \ifx\svgwidth\undefined%
    \setlength{\unitlength}{80.22484197bp}%
    \ifx\svgscale\undefined%
      \relax%
    \else%
      \setlength{\unitlength}{\unitlength * \real{\svgscale}}%
    \fi%
  \else%
    \setlength{\unitlength}{\svgwidth}%
  \fi%
  \global\let\svgwidth\undefined%
  \global\let\svgscale\undefined%
  \makeatother%
  \begin{picture}(1,0.46256357)%
    \lineheight{1}%
    \setlength\tabcolsep{0pt}%
    \put(0,0){\includegraphics[width=\unitlength,page=1]{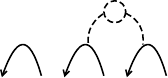}}%
  \end{picture}%
\endgroup%
}$}.
    \end{split}
\end{gather*}
\newcommand\B{\mathcal{B}}
Let $\B$ denote this basis for $A''_2(3)$, and $\B^*$ the dual basis of $\B$ for $A''_2(3)^*$.
Note that $\{u_{123},w_1,w_2,w_3,w_{12},w_{13},w_{23}\}$ is a basis for $A_{2,1}(3)$ and $\{w_1,w_2,w_3,w_{12},w_{13},w_{23}\}$ is a basis for $A_{2,2}(3)$.

\subsection{Generating set of $\Aut(F_3)$}

We have a set of generators of $\Aut(F_3)$ given by Nielsen \cite{Nielsen}, which consists of the following four automorphisms $U,P,\sigma,Q$ called Nielsen transformations:
\begin{align*}
    U(x_1)&=x_1 x_2, & U(x_2)&=x_2,& U(x_3)&=x_3,\\
    P(x_1)&=x_2,     & P(x_2)&=x_1,& P(x_3)&=x_3,\\
    \sigma(x_1)&=x_1^{-1},& \sigma(x_2)&=x_2,& \sigma(x_3)&=x_3,\\
    Q(x_1)&=x_2,& Q(x_2)&=x_3, & Q(x_3)&=x_1.
\end{align*}


\subsection{Action of the Nielsen transformations on $\B$ and $\B^*$}\label{actionofgenerators}

We can compute the action of the Nielsen transformations $U,\sigma, P, Q$ on the basis $\B$ for $A''_2(3)$ and $\B^*$ for $A''_2(3)^*$.
The results are as follows:
$$
\begin{array}{c|c|c|c|c}
     \cdot& U & \sigma & P & Q\\ \hline
     v_1 & v_1+\frac{1}{4} w_{13} & v_1 & v_2 & v_3\\ \hline
     v_2 & v_2-v_1-\frac{3}{2} u_{123}+\frac{1}{4}w_{13}  & -v_2 & v_1 & v_1\\ \hline
     v_3 & v_3+ v_{13} & -v_3 & v_3 & v_2\\ \hline
     v_{12} & v_{12}+\frac{1}{4} w_1+\frac{1}{2}w_{12} & v_{12} & v_{12} & v_{13}\\ \hline
     v_{13} & v_{13} & v_{13} & v_{23} & v_{23}\\ \hline
     v_{23} & v_{23}+ v_{13}+2 v_3 & v_{23} & v_{13} & v_{12}\\ \hline
     u_{123} & u_{123}-\frac{1}{2} w_{13} & -u_{123} & -u_{123} & u_{123}\\ \hline
     w_1 & w_1 & w_1 & w_2 & w_3\\ \hline
     w_2 & w_2+w_1+2 w_{12} & w_2 & w_1 & w_1\\ \hline
     w_3 & w_3 & w_3 & w_3 & w_2\\ \hline
     w_{12} & w_{12}+w_1 & -w_{12} & w_{12} & w_{13}\\ \hline
     w_{13} & w_{13} & -w_{13} & w_{23} & w_{23}\\ \hline
     w_{23} & w_{23}+w_{13} & w_{23} & w_{13} & w_{12}
\end{array}
$$
$$
\begin{array}{c|c|c|c|c}
     \cdot& U & \sigma & P & Q\\ \hline
     v_1^* & v_1^*+ v_2^* & v_1^* & v_2^* & v_3^*\\ \hline
     v_2^* & v_2^* & -v_2^* & v_1^* & v_1^*\\ \hline
     v_3^* & v_3^*-2 v_{23}^* & -v_3^* & v_3^* & v_2^*\\ \hline
     v_{12}^* & v_{12}^* & v_{12}^* & v_{12}^* & v_{13}^*\\ \hline
     v_{13}^* & v_{13}^*-v_3^*+v_{23}^* & v_{13}^* & v_{23}^* & v_{23}^*\\ \hline
     v_{23}^* & v_{23}^* & v_{23}^* & v_{13}^* & v_{12}^*\\ \hline
     u_{123}^* & u_{123}^*+\frac{3}{2} v_2^* & -u_{123}^* & -u_{123}^* & u_{123}^*\\ \hline
     w_1^* & w_1^*+\frac{1}{4}v_{12}^*+w_2^*-w_{12}^* & w_1^* & w_2^* & w_3^*\\ \hline
     w_2^* & w_2^* & w_2^* & w_1^* & w_1^*\\ \hline
     w_3^* & w_3^* & w_3^* & w_3^* & w_2^*\\ \hline
     w_{12}^* & w_{12}^*-\frac{1}{2}v_{12}^*-2w_2^* & -w_{12}^* & w_{12}^* & w_{13}^*\\ \hline
     w_{13}^* & w_{13}^*-\frac{1}{4}v_1^*+\frac{1}{4}v_2^*+\frac{1}{2}u_{123}^*-w_{23}^* & -w_{13}^* & w_{23}^* & w_{23}^*\\ \hline
     w_{23}^* & w_{23}^* & w_{23}^* & w_{13}^* & w_{12}^*
\end{array}
$$

\subsection{Self-duality of $A''_2(3)$}

Define a linear map 
\begin{gather*}
    \Phi: A''_2(3)\to A''_2(3)^*
\end{gather*}
by
\begin{alignat*}{3}
        \Phi(v_1)&=-3 w_{23}^*-\frac{3}{4} v_1^*,&\;
        \Phi(v_2)&=-3 w_{13}^*-\frac{3}{4} v_2^*,&\;
        \Phi(v_3)&=-3 w_{12}^*-\frac{3}{4} v_3^*,\\
        \Phi(v_{12})&=6 w_3^*+\frac{3}{2} v_{13}^*+\frac{3}{2} v_{23}^*,&\;
        \Phi(v_{13})&=6 w_2^*+\frac{3}{2} v_{12}^*+\frac{3}{2} v_{23}^*,&\;
        \Phi(v_{23})&=6 w_1^*+\frac{3}{2} v_{12}^*+\frac{3}{2} v_{13}^*,\\
        &&\;\Phi(u_{123})&= u_{123}^*,&&\\
        \Phi(w_1)&=6 v_{23}^*,&\;
        \Phi(w_2)&=6 v_{13}^*,&\;
        \Phi(w_3)&=6 v_{12}^*,\\
        \Phi(w_{12})&=-3 v_3^*,&\;
        \Phi(w_{13})&=-3 v_2^*,&\;
        \Phi(w_{23})&=-3 v_1^*.
\end{alignat*}
Then we can easily check that $\Phi$ is a linear isomorphism.

\begin{theorem}\label{selfduality}
    The linear isomorphism
    \begin{gather*}
        \Phi: A''_2(3)\xrightarrow{\cong} A''_2(3)^*
    \end{gather*}
    is an isomorphism of $\Aut(F_3)$-modules.
\end{theorem}

\begin{proof}
    We need to check that $\Phi$ is an $\Aut(F_3)$-module map.
    Since $\Aut(F_3)$ is generated by Nielsen transformations,
    we have only to check that 
    \begin{gather*}
        \Phi(x X)=\Phi(x) X
    \end{gather*}
    for $X\in \{U,P,\sigma,Q\}$ and $x\in \B$.
    We can check this by using the tables in Section \ref{actionofgenerators} of the actions of Nielsen transformations on the sets $\B$ and $\B^*$.
    For example, we have
    \begin{gather*}
    \begin{split}
         \Phi(v_1 U)&=\Phi(v_1+\frac{1}{4}w_{13})=\Phi(v_1)+\frac{1}{4}\Phi(w_{13})=-3 w_{23}^*-\frac{3}{4} v_1^*-\frac{3}{4} v_2^*,\\
         \Phi(v_1)U&=(-3 w_{23}^*-\frac{3}{4} v_1^*)U=-3 w_{23}^*-\frac{3}{4} (v_1^*+v_2^*)=-3 w_{23}^*-\frac{3}{4} v_1^*-\frac{3}{4} v_2^*.
    \end{split}
    \end{gather*}
\end{proof}

\begin{corollary}\label{dualA21andA''2/A22}
    The map $\Phi$ induces an isomorphism of $\Aut(F_3)$-modules
    \begin{gather*}
        \overline{\Phi}: A''_2(3)/A_{2,2}(3)\xrightarrow{\cong} A_{2,1}(3)^*.
    \end{gather*}
\end{corollary}

By combining Theorem \ref{thmH1A''2/A22} and Corollary \ref{dualA21andA''2/A22}, we obtain the following.

\begin{corollary}
We have 
$H^1(\IA_3, A_{2,1}(3))\cong V_{2^2,1} \oplus V_{31,1}\oplus V_{21,0}^{\oplus 2}\oplus V_{2,0}$.  
\end{corollary}

\section{The first homology of $\IA_3$ with coefficients in $A''_2(3)$}\label{secH1A''23}

In this section, we will compute the first homology of $\IA_3$ with coefficients in $A''_2(3)$.

\subsection{Generating set of the second homology of $\IA_3$}
\newcommand\rmH{\mathrm{H}}

Here, we recall the result of Day--Putman \cite{DP} on the second homology of $\IA_n$.

\begin{theorem}[Day--Putman \cite{DP}]\label{theoremDP}
    For $n\ge 2$, the second homology $H_2(\IA_n,\Z)$ is a finitely generated $\GL(n,\Z)$-representation, whose generating set consists of elements corresponding to $9$ types of commutator relators of $\IA_n$.
\end{theorem}

If $n=3$, then non-empty sets of commutator relators are $\rmH1, \rmH2, \rmH4, \rmH7, \rmH9$ in \cite[Table 1]{DP} listed below:
\begin{gather}\label{relationsDayPutman}
    \begin{split}
        \rmH1&: \quad [C_{x_a,x_b},C_{x_c,x_b}]=1,\\
        \rmH2&: \quad [M_{x_a^{\alpha},[x_b^{\beta},x_c^{\gamma}]},M_{x_a^{-\alpha},[x_b^{\epsilon},x_c^{\zeta}]^{\xi}}]=1,\\
        \rmH4&: \quad [C_{x_c,x_b}^{\beta}C_{x_a,x_b}^{\beta},C_{x_c,x_a}^{\alpha}]=1,\\
        \rmH7&: \quad [M_{x_c^{\gamma},[x_a^{\alpha},x_b^{\delta}]}, C_{x_a,x_b}^{\beta}][C_{x_c,x_b}^{-\delta}, M_{x_c^{\gamma},[x_a^{\alpha},x_b^{\beta}]}][M_{x_c^{\gamma},[x_a^{\alpha},x_b^{\beta}]} , M_{x_c^{\gamma},[x_a^{\alpha},x_b^{\delta}]}]=1,\\
        \rmH9&: \quad [C_{x_a, x_c}^{\gamma}C_{x_b, x_c}^{\gamma}, C_{x_a, x_b}^{\beta}C_{x_c, x_b}^{\beta}][M_{x_a^{\alpha},[x_b^{\beta},x_c^{\gamma}]},C_{x_b, x_a}^{\alpha}C_{x_c, x_a}^{\alpha}]=1,
    \end{split}
\end{gather}
where $a,b,c$ are distinct elements of $\{1,2,3\}$ and where $\alpha,\beta,\gamma,\delta,\epsilon, \zeta, \xi\in \{\pm 1\}$. 
Here, $C_{x_a,x_b}, M_{x_a^{\alpha},[x_b^{\beta},x_c^{\gamma}]}\in \IA_3$ are defined by
\begin{alignat*}{2}
    C_{x_a,x_b}(x_a)&=x_b x_a x_b^{-1}, &\quad C_{x_a,x_b}(x_i)&=x_i \quad (i\neq a),\\
    M_{x_a^{\alpha},[x_b^{\beta},x_c^{\gamma}]}(x_a^{\alpha})&=[x_b^{\beta},x_c^{\gamma}]x_a^{\alpha},&\quad 
    M_{x_a^{\alpha},[x_b^{\beta},x_c^{\gamma}]}(x_i)&=x_i \quad (i\neq a).
\end{alignat*}
Note that we have $g_{a,b}=C_{x_b,x_a}$ and $f_{a,b,c}=M_{x_c^{-1},[x_b,x_a]}$, where $g_{a,b}$ and $f_{a,b,c}$ are the generators of $\IA_n$ which we recalled in Section \ref{subsecAbeliancycle}.

\begin{remark}
    Day--Putman \cite{DP} remarked that the case where $n=3$ of Theorem \ref{theoremDP} had been obtained by Baker \cite{Baker}.
\end{remark}

We will explicitly write down $2$-cycles corresponding to commutator relators.
Let $G$ be a group.
If we have a commutator relator of the form $[A,B]=1$ for $A,B\in G$, then we can take a $2$-cycle $A\otimes B-B\otimes A \in \Z[G]^{\otimes 2}$ of $G$.
In a similar way, we can take the following $2$-cycle 
\begin{gather*}
    A\otimes B+AB\otimes A^{-1}+ABA^{-1}\otimes B^{-1}+[A,B]\otimes C+[A,B]C\otimes D + [A,B]CD\otimes C^{-1}\\
    -(A\otimes A^{-1}+B\otimes B^{-1}+C\otimes C^{-1})-3(1\otimes 1)\in \Z[G]^{\otimes 2}
\end{gather*}
corresponding to $[A,B][C,D]=1$ for $A,B,C,D\in G$, and the following $2$-cycle 
\begin{gather*}
    A\otimes B+AB\otimes A^{-1}+ABA^{-1}\otimes B^{-1}+[A,B]\otimes C+[A,B]C\otimes D+ [A,B]CD\otimes C^{-1}\\
    +[A,B]CDC^{-1}\otimes D^{-1}+[A,B][C,D]\otimes E+ [A,B][C,D]E\otimes F
    +[A,B][C,D]EF\otimes E^{-1}\\
    -(A\otimes A^{-1}+B\otimes B^{-1}+C\otimes C^{-1}+D\otimes D^{-1}+E\otimes E^{-1})-5(1\otimes 1)\in \Z[G]^{\otimes 2}
\end{gather*}
corresponding to $[A,B][C,D][E,F]=1$ for $A,B,C,D,E,F\in G$. 

\subsection{Highest weight vectors in $H_1(\IA_3, A_{2,2}(3))$}

We have obtained an irreducible decomposition of $H_1(\IA_3, A_{2,2}(3))$ in \eqref{decompH1A22}.
Here, we will observe the highest weight vectors (maximal vectors) of each irreducible component of $H_1(\IA_3, A_{2,2}(3))$.

Benkert--Chakrabarti--Halverson--Leduc--Lee--Stroomer \cite{BCHLLS} gave a basis for the space of maximal vectors of $H^{\otimes p}\otimes (H^*)^{\otimes q}$ for $p,q\ge 0$.
Since $H_1(\IA_3, A_{2,2}(3))\cong \Sym^2 H^*\otimes H^*\otimes {\bigwedge}^{2}H$ can be embedded in $(H^*)^{\otimes 3}\otimes H^{\otimes 2}$ via $\iota'$ appeared in Section \ref{subseccontraction}, we can explicitly give the highest weight vectors of $H_1(\IA_3, A_{2,2}(3))$ as follows.
We take the following basis 
$$\{(e_p^*\cdot e_q^*) \otimes e_r^*\otimes (e_a\wedge e_b)\mid a,b, p,q,r\in \{1,2,3\},\; a<b,\; p\le q\}$$
for $\Sym^2 H^*\otimes H^*\otimes {\bigwedge}^{2}H$.
By considering the lexicographical order on $(e_a\wedge e_b) \otimes (e_p^*\cdot e_q^*)$, set
$$v^i_j=(e_p^*\cdot e_q^*) \otimes e_i^*\otimes (e_a\wedge e_b) \quad (1\le i\le 3,\; 1\le j\le 18),$$
where $(e_a\wedge e_b) \otimes (e_p^*\cdot e_q^*)$ is the $j$-th element.
For example, $v^1_1=(e_1^*\cdot e_1^*)\otimes e_1^*\otimes (e_1\wedge e_2)$, $v^1_2=(e_1^*\cdot e_2^*)\otimes e_1^*\otimes (e_1\wedge e_2)$, $v^1_7=(e_1^*\cdot e_1^*)\otimes e_1^*\otimes (e_1\wedge e_3)$.
Then the highest weight vectors are given in the following lemma.

\begin{lemma}
    Each irreducible component of $H_1(\IA_3, A_{2,2}(3))\cong \Sym^2 H^*\otimes H^*\otimes {\bigwedge}^{2}H$
    has the following highest weight vectors
    \begin{gather*}
            v^3_6\in V_{1^2,3},\quad
            v^2_5+v^2_{12}-v^3_4-v^3_{11}\in V_{1,1^2},\quad
            v^2_6+v^3_{12}\in V_{1,2}^{(1)},\quad
            v^3_5+v^3_{12}\in V_{1,2}^{(2)},\\
            v^1_5+v^1_{12}-v^2_3+v^2_{18}-v^3_9-v^3_{17}\in V_{0,1}.
    \end{gather*}
\end{lemma}

\begin{proof}
    The statement follows from \cite[Theorems 2.5 and 2.7]{BCHLLS}.
\end{proof}

\subsection{The first homology of $\IA_3$ with coefficients in $A_{2,1}(3)$}

We will study the $\GL(3,\Z)$-module structure of $H_1(\IA_3,A_{2,1}(3))$ by using the argument similar to that in Section \ref{secH1A21}.
For $n=3$, however, the boundary map 
$$\partial'_2: H_2(\IA_3, A_{2,1}(3)/A_{2,2}(3))\to H_1(\IA_3, A_{2,2}(3))$$
is not surjective.
We will determine the image of $\partial'_2$ by using the elements of $H_2(\IA_3,\Z)$ corresponding to the relations $\rmH1,\rmH2,\rmH4,\rmH7,\rmH9$ in \eqref{relationsDayPutman} and the highest weight vectors of each irreducible component of $H_1(\IA_3, A_{2,2}(3))$.

\begin{proposition}\label{imageboundary2'n3}
For $n=3$, we have
    \begin{gather*}
        \im(\partial'_2)\cong V_{1,2}\oplus V_{0,1}.
    \end{gather*}
\end{proposition}

\begin{proof}
    The space $A_{2,1}(3)/A_{2,2}(3)$ is one dimensional spanned by $[u_{123}]_{A_{2,2}(3)}$, where $u_{123}$ is described in Section \ref{subsecbasis}.
    Since $\IA_3$ acts trivially on $A_{2,1}(3)/A_{2,2}(3)$, we have $$H_2(\IA_3, A_{2,1}(3)/A_{2,2}(3))\cong A_{2,1}(3)/A_{2,2}(3)\otimes H_2(\IA_3,\Q).$$
    Therefore, by Theorem \ref{theoremDP}, we obtain
    a generating set of $H_2(\IA_3, A_{2,1}(3)/A_{2,2}(3))$ as a $\GL(3,\Z)$-representation.

    One of the $2$-cycles corresponding to the commutator relator $\rmH1$ is
    $$\beta_1=[u_{123}]_{A_{2,2}(3)}\otimes (C_{x_1,x_2}\otimes C_{x_3,x_2}-C_{x_3,x_2}\otimes C_{x_1,x_2}).$$
    We have
    \begin{gather*}
    \begin{split}
        \partial'_2(\beta_1)
        &=[u_{123}, C_{x_1,x_2}]\otimes [C_{x_3,x_2}]_{[\IA_3,\IA_3]}-[u_{123},C_{x_3,x_2}]\otimes [C_{x_1,x_2}]_{[\IA_3,\IA_3]}\\
        &=-v^1_3+v^3_{15}\in V_{1,2}^{(1)}.
    \end{split}
    \end{gather*}
    The other cases in $\rmH1$ can be obtained from $\partial'_2(\beta_1)$ by the action of $\gpS_3$ up to sign.

    One can check that the images of the $2$-cycles corresponding to the commutator relator $\rmH2$ are zero.

    For the $2$-cycle $\beta_4$ corresponding to the commutator relator $$[C_{x_3,x_2}C_{x_1,x_2}, C_{x_3,x_1}]=1,$$
    we have
    \begin{gather*}
    \begin{split}
        \partial'_2(\beta_4)
        &=[u_{123}, C_{x_3,x_2}C_{x_1,x_2}]\otimes [C_{x_3,x_1}]_{[\IA_3,\IA_3]}
        -[u_{123},C_{x_3,x_1}]\otimes [C_{x_3,x_2}C_{x_1,x_2}]_{[\IA_3,\IA_3]}\\
        &=v^1_5-v^3_{17}\in V_{1,2}^{(1)}\oplus V_{0,1}.
    \end{split}
    \end{gather*}
    The other cases can be obtained from $\partial'_2(\beta_4)$ by the action of $\gpS_3$ up to sign.
      
    For the $2$-cycle $\beta_7$ corresponding to the commutator relator $$[M_{x_3,[x_1,x_2]}, C_{x_1,x_2}][C_{x_3,x_2}^{-1}, M_{x_3,[x_1,x_2]}]=1,$$
    we have
    \begin{gather*}
    \begin{split}
        \partial'_2(\beta_7)
        &=[u_{123}, M_{x_3,[x_1,x_2]}]\otimes [C_{x_1,x_2}]_{[\IA_3,\IA_3]}
        +[u_{123}, C_{x_3,x_2}^{-1}]\otimes [M_{x_3,[x_1,x_2]}]_{[\IA_3,\IA_3]}\\
        &+[u_{123},C_{x_1,x_2}]\otimes [M_{x_3,[x_1,x_2]}^{-1}]_{[\IA_3,\IA_3]}
        +[u_{123},M_{x_3,[x_1,x_2]}]\otimes [C_{x_3,x_2}]_{[\IA_3,\IA_3]}\\
        &=v^1_6-v^3_{18}\in V_{1,2}^{(1)}.
    \end{split}
    \end{gather*}
    The other cases can be obtained from $\partial'_2(\beta_7)$ by the action of $\gpS_3$ up to sign.

    For the $2$-cycle $\beta_9$ corresponding to the commutator relator $$[C_{x_1,x_3}C_{x_2,x_3}, C_{x_1,x_2}C_{x_3,x_2}][M_{x_1,[x_2,x_3]}, C_{x_2,x_1}C_{x_3,x_1}]=1,$$
    we have
    \begin{gather*}
    \begin{split}
        \partial'_2(\beta_9)
        =&[u_{123}, C_{x_1,x_3}C_{x_2,x_3}]\otimes [C_{x_1,x_2}C_{x_3,x_2}]_{[\IA_3,\IA_3]}\\
        &+[u_{123}, M_{x_1,[x_2,x_3]}]\otimes [C_{x_2,x_1}C_{x_3,x_1}]_{[\IA_3,\IA_3]}\\
        &+[u_{123},C_{x_1,x_2}C_{x_3,x_2}]\otimes [(C_{x_1,x_3}C_{x_2,x_3})^{-1}]_{[\IA_3,\IA_3]}\\
        &+[u_{123},C_{x_2,x_1}C_{x_3,x_1}]\otimes [M_{x_1,[x_2,x_3]}^{-1}]_{[\IA_3,\IA_3]}\\
        =&-v^2_1-v^3_7\in V_{1,2}^{(1)}\oplus V_{0,1}.
    \end{split}
    \end{gather*}
    The other cases can be obtained from $\partial'_2(\beta_9)$ by the action of $\gpS_3$ up to sign.
    This completes the proof.
\end{proof}

\begin{theorem}\label{thmH1A213}
    We have 
    \begin{gather*}
        H_1(\IA_3,A_{2,1}(3))\cong V_{1^2,3}\oplus V_{1,1^2}\oplus V_{1,2}\oplus V_{0,1^2}.
    \end{gather*}
\end{theorem}

\begin{proof}
    By the same argument as that in the case of $n\ge 4$, we have the isomorphism \eqref{imageboundary1'} and the short exact sequence \eqref{H1A21shortex} for $n=3$.
    By \eqref{decompH1A22} and Proposition \ref{imageboundary2'n3},
    we have
    \begin{gather*}
        \Cok(\partial'_2)\cong  V_{1^2,3}\oplus V_{1,1^2}\oplus V_{1,2}.
    \end{gather*}
    By \eqref{decompH1A21/A22} and \eqref{imageboundary1'}, we have
    \begin{gather*}
        \ker(\partial'_1)\cong V_{0,1^2}.
    \end{gather*}

    Although the short exact sequence  \eqref{H1A21shortex} splits by Lemma \ref{splitalgebraic},
    we can take a section 
    \begin{gather*}
    s: V_{0,1^2}\to H_1(\IA_3,A_{2,1}(3))
    \end{gather*}
    of the projection $H_1(\IA_3,A_{2,1}(3))\twoheadrightarrow \ker(\partial'_1) \cong V_{0,1^2}$, which sends the highest weight vector $(e_1^*\wedge e_2^*\wedge e_3^*)\otimes (e_2^*\otimes (e_1\wedge e_2)+e_3^*\otimes (e_1\wedge e_3))$ of $V_{0,1^2}$ to 
    $[u_{123}\otimes \sigma_{x_1}]$, where $\sigma_{x_1}\in \IA_3$ is defined by $\sigma_{x_1}(x_i)=x_1 x_i x_1^{-1}$ for $i\in \{1,2,3\}$.
    This completes the proof.
\end{proof}

\subsection{The first homology of $\IA_3$ with coefficients in $A''_2(3)$}

In a way similar to the case of $n\ge 4$, we will compute the first homology $H_1(\IA_3,A''_2(3))$.

\begin{theorem}\label{thmH1A2''3}
We have 
\begin{gather*}
H_1(\IA_3,A''_2(3))\cong  V_{1,2^2} \oplus V_{1,31}\oplus V_{0,21}^{\oplus 2}.
\end{gather*}
\end{theorem}
\begin{proof}
    By the same argument as that in the case of $n\ge 4$, we have the isomorphism \eqref{imageboundary''1} and the short exact sequence \eqref{H1A2''shortex}.
    By \eqref{decompH1A''2/A21} and \eqref{imageboundary''1}, we obtain an irreducible decomposition of $\ker(\partial''_1)$, that is, we have $\ker(\partial''_1)\cong V_{1,2^2} \oplus V_{1,31}\oplus V_{0,21}^{\oplus 2}$.

    We will prove that the boundary map 
    $$\partial''_2: H_2(\IA_3, A''_2(3)/A_{2,1}(3))\to H_1(\IA_3, A_{2,1}(3))$$
    is surjective. 
    For the $2$-cycle 
    \begin{gather*}
        \gamma_1=[v_2]_{A_{2,1}(3)}\otimes (C_{x_1, x_2}\otimes C_{x_3,x_2}- C_{x_3,x_2}\otimes C_{x_1,x_2}),
    \end{gather*}
    we have
    \begin{gather*}
        \partial''_2(\gamma_1)
        =3 [u_{123}\otimes \sigma_{x_2}]+\frac{3}{2}w_{13}\otimes (e_3^*\otimes (e_2\wedge e_3)- e_1^*\otimes (e_2\wedge e_1)),
    \end{gather*}
    where $\sigma_{x_2}\in \IA_3$ is defined by $\sigma_{x_2}(x_i)=x_2 x_i x_2^{-1}$ for $i\in \{1,2,3\}$. 
    We can check that $[u_{123}\otimes \sigma_{x_2}]$ generates $V_{0,1^2}\subset H_1(\IA_3,A_{2,1}(3))$ by using the section $s$ appeared in the proof of Theorem \ref{thmH1A213}.
    We can also check that $w_{13}\otimes (e_3^*\otimes (e_2\wedge e_3)- e_1^*\otimes (e_2\wedge e_1))=v^1_3+v^3_{15}$ generates $V_{1^2,3}\oplus V_{1,1^2}\subset H_1(\IA_3,A_{2,1}(3))$ by using the highest weight vectors.
    For the $2$-cycle 
    \begin{gather*}
      \gamma_2=[v_2]_{A_{2,1}(3)}\otimes (M_{x_1,[x_2,x_3]}\otimes M_{x_1^{-1},[x_2,x_3]}- M_{x_1^{-1},[x_2,x_3]}\otimes M_{x_1,[x_2,x_3]}),
    \end{gather*}
    we have
    \begin{gather*}
        \partial''_2(\gamma_2)
        =-3 w_{12}\otimes (e_1^*\otimes (e_2\wedge e_3))
        =-3 v^1_{14}, 
    \end{gather*}
    which generates $V_{1^2,3}\oplus V_{1,2}^{(2)}\subset H_1(\IA_3,A_{2,1}(3))$.
    
    Therefore, by Theorem \ref{thmH1A213}, we have 
    $$\im(\partial''_2)\cong V_{1^2,3}\oplus V_{1,1^2}\oplus V_{1,2}\oplus V_{0,1^2}\cong H_1(\IA_3,A_{2,1}(3))$$
    and $\Cok(\partial''_2)=0$, which completes the proof.
\end{proof}

\section{On higher degree homologies of $\IA_n$ with coefficients in $A_2(n)$}\label{secH2}

In this section, we will study the second and higher degree homologies by considering the relation to the homology of $\Aut(F_n)$ and $\GL(n,\Z)$.

\subsection{The cohomology of $\Aut(F_n)$ with coefficients in $A_{2,1}(n)^*$}

The stable cohomology of $\Aut(F_n)$ with coefficients in $H^{\otimes q}$ was computed by Djament \cite{Djament} and Vespa \cite{Vespa}, and a stable range was given by Randal-Williams \cite{RW}.

\begin{theorem}[Djament \cite{Djament} and Vespa \cite{Vespa}, Randal-Williams \cite{RW}]\label{thmH*Aut}
We have for $n\ge 2i+q+3$,
\begin{gather*}
        H^i(\Aut(F_n),H^{\otimes q})\cong
        \begin{cases}
            0 & i\neq q\\
            \mathcal{P}(q)\otimes S^{1^q} & i=q,
        \end{cases}
\end{gather*}
where $S^{1^q}$ is the sign representation of $\gpS_q$ and $\mathcal{P}(q)$ the permutation module on the set of partitions of $\{1,\dots, q\}$.
\end{theorem}

\begin{lemma}\label{lemHiAutA21}
For $n\ge 2i+6$, we have
    \begin{gather}\label{HiAutA21}
        H^i(\Aut(F_n),A_{2,1}(n)^*)\cong 
        \begin{cases}
            \Q^{\oplus 3} & i=3\\
            0 & \text{otherwise}.
        \end{cases}
    \end{gather}
\end{lemma}

\begin{proof}
By Theorem \ref{thmH*Aut}, we have \begin{gather*}
    H^i(\Aut(F_n), A_{2,2}(n)^*)\cong H^i(\Aut(F_n), V_{2,0})=0 
\end{gather*}
for $n\ge 2i+5$ and 
\begin{gather*}
    H^i(\Aut(F_n),(A_{2,1}(n)/A_{2,2}(n))^*)\cong
    \begin{cases}
        \Q^{\oplus 3} & i=3\\
        0 &  \text{otherwise}
    \end{cases}
\end{gather*}
for $n\ge 2i+6$.
Therefore, we obtain \eqref{HiAutA21} by using the long exact sequence 
\begin{gather*}
\begin{split}
\cdots& \to H^{i-1}(\Aut(F_n), A_{2,2}(n)^*)\to H^i(\Aut(F_n), (A_{2,1}(n)/A_{2,2}(n))^*)\\
&\to  H^i(\Aut(F_n), A_{2,1}(n)^*)\to
H^i(\Aut(F_n), A_{2,2}(n)^*)\\
&\to H^{i+1}(\Aut(F_n), (A_{2,1}(n)/A_{2,2}(n))^*)\to \cdots
\end{split}
\end{gather*}
associated to the short exact sequence
\begin{gather*}
    0\to A_{2,2}(n)\to A_{2,1}(n)\to A_{2,1}(n)/A_{2,2}(n)\to 0.
\end{gather*}
\end{proof}

\subsection{On the $\GL(n,\Z)$-invariant part of $H^*(\IA_n,A_{2,1}(n)^*)$}

Here we study the $\GL(n,\Z)$-invariant part of $H^*(\IA_n, A_{2,1}(n)^*)$ by using 
the Hochschild--Serre spectral sequences
\begin{gather}\label{spseq}
    E_2^{p,q}=H^p(\GL(n,\Z),H^q(\IA_n,A_{2,1}(n)^*))\Rightarrow H^{p+q}(\Aut(F_n),A_{2,1}(n)^*)
\end{gather}
associated to the short exact sequence
$$
1\to \IA_n\to \Aut(F_n)\to \GL(n,\Z)\to 1.
$$

Before stating our result, we will recall Borel's stability and vanishing theorem.

\begin{theorem}[Borel \cite{Borel1, Borel2}, Li-Sun \cite{Li-Sun} and Kupers--Miller--Patzt \cite{KMP}]\label{thmBorel}
    We have
    \begin{gather*}
        H^*(\GL(n,\Z),\Q)\cong {\bigwedge}_{\Q} (x_1,x_2,\cdots),\quad \deg (x_i)=4i+1 
    \end{gather*}
    for $*\le n-1$
    and for any non-trivial algebraic $\GL(n,\Z)$-representation $V_{\underline{\lambda}}$,
    \begin{gather*}
        H^*(\GL(n,\Z),V_{\underline{\lambda}})=0
    \end{gather*}
    for $*\le n-2$ .
\end{theorem}

We obtain the vanishing of the $\GL(n,\Z)$-invariant part of the second cohomology of $\IA_n$ with coefficients in $A_{2,1}(n)^*$.

\begin{proposition}\label{propH2coinvA21}
For $n\ge 10$, we have 
\begin{gather*}
    H^2(\IA_n,A_{2,1}(n)^*)^{\GL(n,\Z)}=0.
\end{gather*}    
\end{proposition}

\begin{proof}
    Since we have $H^0(\IA_n,A_{2,1}(n)^*)\cong V_{1^3,0}$, we have
    $E_2^{p,0}=0$ for $n\ge p+1$ by Theorem \ref{thmBorel}.
    By Theorems \ref{thmH1A21}, \ref{thmH1A213} and \ref{thmBorel}, we have $E_2^{p,1}=0$
    for $n\ge p+2$.
    By Lemma \ref{lemHiAutA21}, we have
    $H^2(\Aut(F_n),A_{2,1}(n)^*)=0$ for $n\ge 10$.
    Therefore, by using the spectral sequence \eqref{spseq}, we obtain
    $H^2(\IA_n,A_{2,1}(n)^*)^{\GL(n,\Z)}=0$ for $n\ge 10$.
\end{proof}

\begin{remark}
By the spectral sequence \eqref{spseq}, we also obtain 
\begin{gather}\label{H3coinvA21}
        E_2^{1,2}\oplus \ker(d_2: E_2^{0,3}\to E_2^{2,2})\cong \Q^{\oplus 3}
\end{gather}
for $n\ge 12$.
\end{remark}

In what follows, we will study the $\GL(n,\Z)$-invariant part of $H^i(\IA_n,A_{2,1}(n)^*)$ for $i\ge 3$ under the hypothesis that the $\GL(n,\Z)$-representation $H^i(\IA_n,\Q)$ is stably algebraic.
By Lemma \ref{splitalgebraic}, if $H^i(\IA_n,\Q)$ is algebraic, then so is $H^i(\IA_n,A_{2,1}(n)^*)$.    

\begin{proposition}
    Suppose that there exists $n_2\ge 0$ such that $H^2(\IA_n,\Q)$ is algebraic for $n\ge n_2$. 
    Then we have
    $$H^3(\IA_n,A_{2,1}(n)^*)^{\GL(n,\Z)}\cong \Q^{\oplus 3}$$
    for $n\ge \max(12, n_2)$.

    Moreover, if for any $i\ge 2$, there exist $n_i\ge 0$ such that $H^i(\IA_n,\Q)$ are algebraic for $n\ge n_i$,
    then we have
    \begin{gather*}
        H^i(\IA_n,A_{2,1}(n)^*)^{\GL(n,\Z)}\cong (H^{i-3}(\IA_n,\Q)^{\GL(n,\Z)})^{\oplus 3}
    \end{gather*}
    for $n\ge \max(2i+6, n_2,\dots, n_i)$.
\end{proposition}

\begin{proof}
    The first statement follows from \eqref{H3coinvA21}, Theorem \ref{thmBorel} and Proposition \ref{propH2coinvA21}.
    The second statement follows in an argument similar to that of the proof of \cite[Lemma 7.7]{HK}.
\end{proof}

It seems natural to propose the following conjecture.
By \emph{polynomial $\Aut(F_n)$-modules}, we mean $\Aut(F_n)$-modules that are obtained from algebraic $\GL(n,\Z)$-representations by extensions.
The spaces of Jacobi diagrams are examples of polynomial $\Aut(F_n)$-modules.

\begin{conjecture}\label{conjcoinv}
For any polynomial $\Aut(F_n)$-module $M$, we have
    \begin{gather*}
        H^*(\Aut(F_n),M)\otimes H^*(\IA_n,\Q)^{\GL(n,\Z)}\cong H^*(\IA_n,M)^{\GL(n,\Z)}
    \end{gather*}
for sufficiently large $n$ with respect to the cohomological degree.
\end{conjecture}

In the case of $H^{\otimes p}\otimes (H^*)^{\otimes q}$, Conjecture \ref{conjcoinv} follows from our previous conjectures \cite[Conjecture 12.6]{KatadaIA} and \cite[Conjecture 7.4]{HK}.

\begin{remark}
The inclusion map $\iota: \IA_n\hookrightarrow \Aut(F_n)$ induces a homomorphism
$$\iota^i: H^i(\Aut(F_n),M)\to  H^i(\IA_n,M)^{\GL(n,\Z)}.$$
On the other hand, we have the cup product map for cohomology
$$
\cup: H^i(\IA_n,M)^{\GL(n,\Z)}\otimes H^j(\IA_n,\Q)^{\GL(n,\Z)}\to H^{i+j}(\IA_n,M)^{\GL(n,\Z)}.
$$
Therefore, we have a homomorphism
$$
\xi=\cup \circ (\iota^i \otimes \id) :H^i(\Aut(F_n),M)\otimes  H^j(\IA_n,\Q)^{\GL(n,\Z)}\to
H^{i+j}(\IA_n,M)^{\GL(n,\Z)}.
$$
By using an argument similar to that of the proof of \cite[Lemma 7.7]{HK}, we can prove that if $H^*(\IA_n,\Q)$ is stably algebraic, then the map $\xi$ is an isomorphism for sufficiently large $n$ with respect to $i+j$, and therefore, Conjecture \ref{conjcoinv} holds.
\end{remark}

\subsection{On the $\GL(n,\Z)$-invariant part of $H^*(\IA_n,A''_2(n)^*)$}

In a way similar to Lemma \ref{lemHiAutA21}, for $n\ge 2i+7$, we obtain 
\begin{gather*}
     H^i(\Aut(F_n),A''_2(n)^*)=0, \quad i\neq 3,4,
\end{gather*}
and an exact sequence
\begin{gather*}
    0\to H^3(\Aut(F_n),A''_2(n)^*)\to H^3(\Aut(F_n),A_{2,1}(n)^*)\cong \Q^{\oplus 3}\\
    \to H^4(\Aut(F_n),(A''_2(n)/A_{2,1}(n))^*)\cong \Q^{\oplus 2}\to H^4(\Aut(F_n),A''_2(n)^*)\to 0.
\end{gather*}

\begin{problem}
    Determine $H^3(\Aut(F_n),A''_2(n)^*)$ and $H^4(\Aut(F_n),A''_2(n)^*)$.
\end{problem}

In a way similar to Proposition \ref{propH2coinvA21}, we obtain the vanishing of the $\GL(n,\Z)$-invariant part of the second cohomology of $\IA_n$ with coefficients in $A''_2(n)^*$.

\begin{proposition}
    For $n\ge 11$, we have
    \begin{gather*}
        H^2(\IA_n,A''_2(n)^*)^{\GL(n,\Z)}=0.
    \end{gather*}
\end{proposition}


\subsection{Some quotient representation of $H_2(\IA_n,A_{2,1}(n))$}

Here we observe some quotient space of $H_2(\IA_n,A_{2,1}(n))$ by using the second Albanese homology of $\IA_n$, which was determined by Pettet \cite{Pettet}.
In particular, $H_2(\IA_n,A_{2,1}(n))$ is non-trivial.

\begin{lemma}[Pettet \cite{Pettet}]\label{lemPettet}
    We have
    \begin{gather*}
        H^A_2(\IA_n,\Q)\cong
        \begin{cases}
            V_{1^4,1^2}\oplus V_{21^2,2}\oplus V_{2^2,1^2}\oplus V_{21,1}\oplus V_{1^3,1}^{\oplus 2}\oplus V_{1^2,0}^{\oplus 2} & n\ge 6\\
            V_{21^2,2}\oplus V_{2^2,1^2}\oplus V_{21,1}\oplus V_{1^3,1}^{\oplus 2}\oplus V_{1^2,0}^{\oplus 2} & n= 5\\
            V_{21^2,2}\oplus V_{2^2,1^2}\oplus V_{21,1}\oplus V_{1^3,1}\oplus V_{1^2,0}^{\oplus 2} & n=4\\
            V_{21,1}\oplus  V_{1^2,0} & n=3.
        \end{cases}
    \end{gather*}
\end{lemma}

\begin{proposition}\label{propH2A21}
For $n\ge 9$, we have a quotient $\GL(n,\Z)$-representation of $H_2(\IA_n,A_{2,1}(n))$ that is isomorphic to
    \begin{gather}\label{H2A21Alb}
        \begin{split}
        &V_{1^4,1^5}\oplus V_{1^4,21^3}\oplus V_{1^4,2^21}
        \oplus V_{21^2,21^3} \oplus V_{21^2,31^2}
        \oplus V_{2^2,1^5}\oplus V_{2^2,21^3}\oplus V_{2^2,2^21}\\
        &\oplus V_{1^3,1^4}^{\oplus 3}\oplus V_{1^3,21^2}^{\oplus 4}\oplus V_{1^3,2^2}\oplus V_{1^3,31}
        \oplus V_{21,1^4}^{\oplus 2}\oplus V_{21,21^2}^{\oplus 3}\oplus V_{21,2^2}\oplus V_{21,31}\\
        &\oplus V_{1^2,1^3}^{\oplus 5}\oplus V_{1^2,21}^{\oplus 5}
        \oplus V_{2,1^3}\oplus V_{2,21}^{\oplus 2}\oplus V_{2,3}
        \oplus V_{1,1^2}^{\oplus 3}\oplus V_{1,2}^{\oplus 2}\oplus V_{0,1}\\
        &\oplus V_{0,1^5}^{\oplus 2}\oplus V_{0,21^3}^{\oplus 2}\oplus V_{0,2^21}^{\oplus 2}.
        \end{split}
    \end{gather}
\end{proposition}

\begin{proof}
By the long exact sequence \eqref{H1A21}, $\ker (\partial'_2)$ is a quotient $\GL(n,\Z)$-representation of $H_2(\IA_n,A_{2,1}(n))$.
Since $\partial'_2$ is surjective by Lemma \ref{imageboundary2'}, we have the short exact sequence
\begin{gather*}
    0\to \ker(\partial'_2) \to  H_2(\IA_n, A_{2,1}(n)/A_{2,2}(n))\xrightarrow{\partial'_2} H_1(\IA_n,A_{2,2}(n))\to 0.
\end{gather*}
Since $\IA_n$ acts trivially on $A_{2,1}(n)/A_{2,2}(n)$, we have a surjective homomorphism of $\GL(n,\Z)$-representations
\begin{gather*}
     H_2(\IA_n, A_{2,1}(n)/A_{2,2}(n))\twoheadrightarrow A_{2,1}(n)/A_{2,2}(n)\otimes H^A_2(\IA_n,\Q).
\end{gather*}
By Lemma \ref{lemPettet}, we obtain for $n\ge 9$,
\begin{gather*}
    \begin{split}
        &A_{2,1}(n)/A_{2,2}(n)\otimes H^A_2(\IA_n,\Q)\\
        &\cong
        V_{1^4,1^5}\oplus V_{1^4,21^3}\oplus V_{1^4,2^21}
        \oplus V_{21^2,21^3} \oplus V_{21^2,31^2}
        \oplus V_{2^2,1^5}\oplus V_{2^2,21^3}\oplus V_{2^2,2^21}\\
        &\oplus V_{1^3,1^4}^{\oplus 3}\oplus V_{1^3,21^2}^{\oplus 4}\oplus V_{1^3,2^2}\oplus V_{1^3,31}
        \oplus V_{21,1^4}^{\oplus 2}\oplus V_{21,21^2}^{\oplus 3}\oplus V_{21,2^2}\oplus V_{21,31}\\
        &\oplus V_{1^2,1^3}^{\oplus 5}\oplus V_{1^2,21}^{\oplus 6}\oplus V_{1^2,3}
        \oplus V_{2,1^3}\oplus V_{2,21}^{\oplus 2}\oplus V_{2,3}
        \oplus V_{1,1^2}^{\oplus 4}\oplus V_{1,2}^{\oplus 4}
        \oplus V_{0,1}^{\oplus 2}\\
        &\oplus V_{0,1^5}^{\oplus 2}\oplus V_{0,21^3}^{\oplus 2}\oplus V_{0,2^21}^{\oplus 2}.
    \end{split}
\end{gather*}
Therefore, by the computation \eqref{decompH1A22}, we obtain the quotient $\GL(n,\Z)$-representation of $H_2(\IA_n,A_{2,1}(n))$ that is isomorphic to \eqref{H2A21Alb}.
\end{proof}

\begin{remark}
We obtain some non-trivial quotient $\GL(n,\Z)$-representation of $H_2(\IA_n,A_{2,1}(n))$ for $4\le n\le 8$ as well.

For $n=3$, we have $\im (\partial'_2)\cong A_{2,1}(3)/A_{2,2}(3)\otimes H^A_2(\IA_3,\Q)$.
However, by the work of Satoh \cite{Satoh2021}, we have a $\GL(3,\Z)$-subrepresentation of $H^2(\IA_3,\Q)$ which is not contained in $H^2_A(\IA_3,\Q)$.
Therefore, we obtain the non-triviality of $H_2(\IA_n,A_{2,1}(3))$. 
\end{remark}

\begin{remark}
In a similar way, we can also compute some quotient vector space of $H_2(\IA_n,A''_2(n))$ and prove the non-triviality.
\end{remark}

\section{The homology of $\IO_n$}\label{secH1IO}

Let $\IO_n$ denote the kernel of the canonical map from the outer automorphism group $\Out(F_n)$ of $F_n$ to $\GL(n,\Z)$.
By Lemma \ref{Innerautomorphism}, the $\Aut(F_n)$-action on $A_d(n)$ induces an action of $\Out(F_n)$ on $A_d(n)$.
In this section, we will study the homology of $\IO_n$ with coefficients in $A_2(n)$.

\subsection{Hochschild--Serre spectral sequence}

We have a short exact sequence of groups
\begin{gather}\label{exactIAIO}
    1\to \Inn(F_n) \to \IA_n \to \IO_n\to 1.
\end{gather}
Let $M$ be a right $\IA_n$-module such that the restriction is a trivial $\Inn(F_n)$-module.
The Hochschild--Serre spectral sequence associated to 
\eqref{exactIAIO} is
\begin{gather}\label{spseqIO}
E^2_{p,q}=H_p(\IO_n,H_q(\Inn(F_n),M))\Rightarrow H_{p+q}(\IA_n,M).
\end{gather}
For $n\ge 3$, we have 
\begin{gather*}
    H_i(\Inn(F_n),M)\cong 
    \begin{cases}
        M & i=0\\
        M\otimes H & i=1\\
        0 & \text{otherwise}
    \end{cases}
\end{gather*}
since we have $\Inn(F_n)\cong F_n$ and $M$ is a trivial $\Inn(F_n)$-module.

We obtain the surjectivity of the map between the first homology of $\IA_n$ and $\IO_n$.

\begin{lemma}\label{surjectivehomfromIAtoIO}
The $\GL(n,\Z)$-homomorphism
\begin{gather*}
    H_1(\IA_n,M)\to H_1(\IO_n,M)
\end{gather*}   
induced by the projection $\IA_n\to \IO_n$ is surjective.
\end{lemma}

\begin{proof}
    By using the spectral sequence \eqref{spseqIO}, we obtain the surjectivity since we have $H_1(\IO_n,M)=E^2_{1,0}=E^{\infty}_{1,0}$.
\end{proof}

\subsection{Computations of the first homology of $\IO_n$ with coefficients in $A_2(n)$}

By \cite[Theorem 6.2]{Kawazumi} (see also \cite[Theorem 2]{Formanek}), we have the following isomorphism of $\GL(n,\Z)$-representations
\begin{gather*}
    H_1(\IO_n,\Z)\cong \Hom(H_{\Z},{\bigwedge}^2 H_{\Z})/H_{\Z}.
\end{gather*}
By tensoring with $\Q$, we have the $\GL(n,\Z)$-representation structure of the first rational homology of $\IO_n$
\begin{gather*}
    H_1(\IO_n,\Q)\cong \Hom(H,{\bigwedge}^2 H)/H.
\end{gather*}

As in Section \ref{irdecompH1}, we have the following irreducible decomposition of the first homology of $\IO_n$ with trivial coefficients:
\begin{gather*}
    H_1(\IO_n,A'_2(n))\cong 
    \begin{cases}
        V_{1^2,41}\oplus V_{1^2,5}\oplus V_{1,4}\oplus V_{1,31} & n\ge 4\\
        V_{1^2,5}\oplus V_{1,4}\oplus V_{1,31} & n= 3,
    \end{cases}
\end{gather*}
\begin{gather*}
\begin{split}
    &H_1(\IO_n, A''_2(n)/A_{2,1}(n))\\
    &\cong 
    \begin{cases}
        V_{1^2,2^21}\oplus V_{1^2,32}\oplus V_{1,21^2}\oplus V_{1,2^2} \oplus V_{1,31}\oplus V_{0,1^3}\oplus V_{0,21} & n\ge 5\\
        V_{1^2,32}\oplus V_{1,21^2}\oplus V_{1,2^2} \oplus V_{1,31}\oplus V_{0,1^3}\oplus V_{0,21} & n= 4\\
        V_{1,31}\oplus V_{0,1^3}\oplus V_{0,21} & n=3,
    \end{cases}
\end{split}
\end{gather*}
\begin{gather*}
\begin{split}
      &H_1(\IO_n, A_{2,1}(n)/A_{2,2}(n))\\
      &\cong 
      \begin{cases}
          V_{1^2,1^4}\oplus V_{1^2,21^2}\oplus V_{1,1^3}\oplus V_{1,21}\oplus V_{0,1^2}\oplus V_{0,2} & n\ge 6\\
          V_{1^2,21^2}\oplus V_{1,1^3}\oplus V_{1,21}\oplus V_{0,1^2}\oplus V_{0,2} & n=5\\
          V_{1,21}\oplus V_{0,1^2}\oplus V_{0,2} & n=4\\
          V_{0,2} & n=3,
      \end{cases}    
\end{split}
\end{gather*}
\begin{gather*}
 \begin{split}
    H_1(\IO_n,A_{2,2}(n))
    \cong 
        \begin{cases}
            V_{1^2,21}\oplus V_{1^2,3}\oplus V_{1,1^2}\oplus V_{1,2} & n\ge 4\\
            V_{1^2,3}\oplus V_{1,1^2}\oplus V_{1,2} & n=3.
        \end{cases}     
 \end{split}
\end{gather*}

By using the above computation of irreducible decompositions and Lemma \ref{surjectivehomfromIAtoIO}, and an argument similar to the case of $\IA_n$, we obtain the first homology of $\IO_n$ with coefficients in the spaces of Jacobi diagrams.

\begin{theorem}\label{thmH1IOA''2A22}
    We have
    \begin{gather*}
    \begin{split}
       H_1(\IO_n,A''_2(n)/A_{2,2}(n))
       \cong
        \begin{cases}
        V_{1^2,2^21}\oplus V_{1^2,32}\oplus V_{1,21^2}\oplus V_{1,2^2} \oplus V_{1,31}
        \oplus V_{0,21}\oplus V_{0,2} & n\ge 5\\
        V_{1^2,32}\oplus V_{1,21^2}\oplus V_{1,2^2} \oplus V_{1,31}\oplus V_{0,21}\oplus V_{0,2} & n= 4\\
        V_{1,31}\oplus V_{0,21}\oplus V_{0,2} & n=3.
        \end{cases}
    \end{split}
    \end{gather*}
\end{theorem}
\begin{proof}
    The proof is similar to that of Theorem \ref{thmH1A''2/A22}.
\end{proof}

\begin{theorem}\label{thmH1IOA21}
We have
\begin{gather*}
H_1(\IO_n,A_{2,1}(n))\cong 
        \begin{cases}
          V_{1^2,1^4}\oplus V_{1^2,21^2}\oplus V_{1,1^3}\oplus V_{1,21}\oplus V_{0,1^2} & n\ge 6\\
          V_{1^2,21^2}\oplus V_{1,1^3}\oplus V_{1,21}\oplus V_{0,1^2} & n=5\\
          V_{1,21}\oplus V_{0,1^2} & n=4\\
          V_{1^2,3}\oplus V_{1,1^2}\oplus V_{1,2} & n=3.
        \end{cases}
\end{gather*}
\end{theorem}
\begin{proof}
    For $n\ge 4$, the proof is similar to that of Theorem \ref{thmH1A21}.
    
    For $n=3$, corresponding to \eqref{H1A21shortex},
    we have the short exact sequence
    $$0\to \Cok(\partial'_2)\to H_1(\IO_3,A_{2,1}(3))\to \ker (\partial'_1)=0\to 0,$$
    where $\Cok(\partial'_2)$ is a quotient representation of $H_1(\IO_3, A_{2,2}(3))\cong V_{1^2,3}\oplus V_{1,1^2}\oplus V_{1,2}$.
    
    On the other hand, 
    $E^{\infty}_{0,1}=\Cok (d^2_{2,0}: E^2_{2,0}\to E^2_{0,1})$ is a quotient representation of
    $E^2_{0,1}=H_0(\IO_3,H\otimes A_{2,1}(3))\cong V_{0,1^2}$.
    Since we have $H_1(\IA_3, A_{2,1}(3))\cong V_{1^2,3}\oplus V_{1,1^2}\oplus V_{1,2}\oplus V_{0,1^2}$ by Theorem \ref{thmH1A213},
    we obtain $H_1(\IO_n,A_{2,1}(3))\cong V_{1^2,3}\oplus V_{1,1^2}\oplus V_{1,2}$.
\end{proof}

\begin{theorem}\label{thmH1IOA''2}
We have
\begin{gather*}
  \begin{split}
        H_1(\IO_n, A''_2(n))\cong 
        \begin{cases}
        V_{1^2,2^21}\oplus V_{1^2,32}\oplus V_{1,21^2}\oplus V_{1,2^2} \oplus V_{1,31}\oplus V_{0,21} & n\ge 5\\
        V_{1^2,32}\oplus V_{1,21^2}\oplus V_{1,2^2} \oplus V_{1,31}\oplus V_{0,21} & n= 4\\
        V_{1,31}\oplus V_{0,21}& n=3.
        \end{cases}
  \end{split} 
\end{gather*}
\end{theorem}
\begin{proof}
    For $n\ge 4$, the proof is similar to that of Theorem \ref{thmH1A2''}.

    For $n=3$, correponding to \eqref{H1A2''shortex}, we have
    $$
    0\to \Cok(\partial''_2)\to H_1(\IO_3,A''_2(3))\to \ker (\partial''_1)\cong V_{1,31}\oplus V_{0,21}\to 0,
    $$
    where $\Cok(\partial''_2)$ is a quotient representation of $H_1(\IO_3, A_{2,1}(3))\cong V_{1^2,3}\oplus V_{1,1^2}\oplus V_{1,2}$.

    On the other hand, $E^{\infty}_{0,1}=\Cok (d^2_{2,0}: E^2_{2,0}\to E^2_{0,1})$ is a quotient representation of
    $E^2_{0,1}=H_0(\IO_3,H\otimes A''_2(3))\cong V_{1,2^2}\oplus V_{0,21}$.
    Since we have $H_1(\IA_3, A''_2(3))\cong V_{1,2^2}\oplus V_{1,31}\oplus V_{0,21}^{\oplus 2}$ by Theorem \ref{thmH1A2''3},
    we obtain $H_1(\IO_n,A''_2(3))\cong V_{1,31}\oplus V_{0,21}$.
\end{proof}

By Theorems \ref{thmH1IOA''2A22}, \ref{thmH1IOA21} and \ref{thmH1IOA''2}, the surjectivity of the map between the second homology of $\IA_n$ and $\IO_n$ follows from an argument similar to Lemma \ref{surjectivehomfromIAtoIO}.

\begin{proposition}
Let $M= A''_2(n)/A_{2,2}(n), A_{2,1}(n), A''_2(n)$.
Then the $\GL(n,\Z)$-homomorphism 
\begin{gather*}
    H_2(\IA_n,M)\to H_2(\IO_n,M)
\end{gather*}
induced by the projection $\IA_n\to \IO_n$ is surjective.   
\end{proposition}

\begin{remark}
We have $d^2_{2,0}=0$ in the spectral sequence \eqref{spseqIO}.
If the differentials $d^2_{i,0}$ and $d^2_{i+1,0}$ are zero, then we obtain an isomorphism of $\GL(n,\Z)$-representations
\begin{gather*}
        \gr H_i(\IA_n,M)\cong H_i(\IO_n,M)\oplus H_{i-1}(\IO_n,H\otimes M),
\end{gather*}
where $\gr H_i(\IA_n,M)$ denotes the graded $\GL(n,\Z)$-representation associated to the filtration of $H_i(\IA_n,M)$ of length $2$.
If moreover, $H_i(\IA_n,M)$ is algebraic, then $H_i(\IO_n,M)$ is also algebraic and we obtain an isomorphism of $\GL(n,\Z)$-representations
\begin{gather*}
H_i(\IA_n,M)\cong H_i(\IO_n,M)\oplus H_{i-1}(\IO_n,H\otimes M).
\end{gather*}
\end{remark}

\section{Perspectives}\label{secperspectives}

In this section, we will give some perspectives on the first homology of $\IA_n$ with coefficients in the spaces $A_d(n)$ of Jacobi diagrams of $d$ not smaller than $3$.

In this paper, we have used the long exact sequences of homology associated to the short exact sequences of coefficients.
By using the spectral sequence associated to the filtered chain complex 
\begin{gather*}
   0\subset C_*(\IA_n,A_{2,2}(n))\subset C_*(\IA_n,A_{2,1}(n))\subset C_*(\IA_n,A''_2(n)),
\end{gather*}
we can directly compute the $\GL(n,\Z)$-module structure of the associated graded of $H_1(\IA_n,A''_2(n))$ with respect to the filtration induced by the above filtration of chain complexes.
We expect that we can use this strategy to compute the first homology of $\IA_n$ with coefficients in $A_d(n)$ with $d\ge 3$.

\bibliographystyle{plain}
\bibliography{reference.bib}

\end{document}